\documentclass[12pt, reqno]{amsart}
\usepackage{hyperref, bbm, amssymb, mathtools, url}
\usepackage{mathrsfs}
\usepackage{amsmath,stmaryrd}
\usepackage{cleveref}
\usepackage{fullpage}
\usepackage{tikz}
\usepackage[alphabetic]{amsrefs}
\usetikzlibrary{matrix,arrows,decorations.pathmorphing, cd}
\usepackage{comment}

\newcommand{\qednow}{\pushQED{\qed}\qedhere\popQED}

\newcommand{\Hom}       {\operatorname{Hom}}

\newcommand{\Fun}       {\operatorname{Fun}}
\newcommand{\Glob}       {\operatorname{Glob}}

\newcommand{\op}        {\operatorname{op}}
\newcommand{\laxlim}    {\operatorname*{laxlim}}
\newcommand{\laxlimdag}    {\operatorname*{laxlim^\dagger}}

\newcommand{\oplaxlim}    {\operatorname*{oplaxlim}}

\newcommand{\Stab}      {\operatorname{Stab}}

\newcommand{\Spc}{\mathcal{S}}          
\newcommand{\Ar}{\operatorname{Ar}}     
\newcommand{\ul}[1]{\underline{#1}}
\newcommand{\aug}{\mathrm{aug}}
\newcommand{\GammaS}{{\Gamma\Spc}}
\newcommand{\mySp}{{\mathscr S\kern-1.75ptp}}

\newcommand{\CB}        {{\mathbf{B}}}
\newcommand{\CC}        {{\mathcal{C}}}
\newcommand{\CD}        {{\mathcal{D}}}
\newcommand{\CE}        {{\mathcal{E}}}

\newcommand{\CI}        {{\mathcal{I}}}
\newcommand{\CJ}        {{\mathcal{J}}}

\newcommand{\CL}        {{\mathcal{L}}}

\newcommand{\CP}        {{\mathcal{P}}}

\newcommand{\CV}        {{\mathcal{V}}}
\newcommand{\CW}        {{\mathcal{W}}}
\newcommand{\CF}        {{\mathcal{F}}}

\renewcommand{\CW}      {{\mathcal{W}}}

\newcommand{\bB}        {{\mathbf{B}}}

\newcommand{\bbL}        {{\mathbb{L}}}
\newcommand{\bbR}        {{\mathbb{R}}}
\newcommand{\bbI}        {{\mathbb{I}}}
\newcommand{\bbF}        {{\mathbb{F}}}
\newcommand{\bbS}        {{\mathbb{S}}}

\newcommand{\Mod}{\operatorname{Mod}}   
\newcommand{\CAlg}{\operatorname{CAlg}} 

\newcommand{\Glo}       {\mathrm{Glo}}
\newcommand{\Orb}       {\mathrm{Orb}}

\newcommand{\Rep}		{\mathrm{Rep}}
\newcommand{\rep}		{\mathrm{rep}}

\newcommand{\Sp}        {\mathrm{Sp}}

\newcommand{\Cat}       {\mathrm{Cat}}

\newcommand{\PrLun}     {{\mathsf{Pr}^{\L}}}
\newcommand{\PrL}[2]       {{\mathsf{Pr}^{#2}_{#1}}}
\newcommand{\uPrL}[2]       {{\ul{\mathsf{Pr}}^{#2}_{#1}}}

\newcommand{\PrLVst}[2]       {{\mathsf{Pr}^{#2}_{#1,\textup{rep-st}}}}
\newcommand{\PrLast}[2]       {{\mathsf{Pr}^{#2}_{#1,\ast}}}
\newcommand{\PrLotimes}[2]	     {{\mathsf{PM}_{#1}^{#2}}}

\newcommand{\PrLVstotimes}[2]       {{\mathsf{PM}^{#2}_{#1,\textup{rep-st}}}}

\newcommand{\PrLPsemi}[2]			{{\mathrm{Pr}^{#2}_{#1,P\text{-}\oplus}}}
\newcommand{\PrLPst}[2]			{{\mathrm{Pr}^{#2}_{#1,P\textup{-}\textup{st}}}}

\renewcommand{\L}			{\mathrm{L}}
\newcommand{\st}		{\mathrm{st}}
\newcommand{\gl}		{{\mathrm{gl}}}

\newcommand{\id}        {\mathrm{id}}

\newcommand{\Free}       {{\mathcal{P}_{\scriptstyle{S}}^{\scriptstyle{T}}}}
\newcommand{\FreeOrb}       {{\mathcal{P}_\Orb^\Glo}}
\newcommand{\Spbulgl}		{\Sp_{\bullet\text{-}\mathrm{gl}}}
\newcommand{\Spcbulgl}		{\Spc_{\bullet\text{-}\mathrm{gl}}}

\newcommand{\fgt}		{\mathrm{fgt}}

\newcommand{\ev}		{\mathrm{ev}}

\DeclareMathOperator{\Nm}{Nm}

\DeclareMathOperator{\Nmadjdual}{\overline{\Nm}}

\DeclareMathOperator{\core}{core}

\newcommand{\PSh}		{{\mathrm{PSh}}}

\newcommand{\colim}  		{\operatornamewithlimits{colim}}
\newcommand{\cocolon}{\nobreak \mskip6mu plus1mu \mathpunct{}\nonscript\mkern-\thinmuskip {:}\mskip2mu \relax}
\newcommand{\Unco}[1]       {\mathrm{Un}^\mathrm{co}(#1)}

\usetikzlibrary{calc}
\tikzset{curve/.style={settings={#1},to path={(\tikztostart)
			.. controls ($(\tikztostart)!\pv{pos}!(\tikztotarget)!\pv{height}!270:(\tikztotarget)$)
			and ($(\tikztostart)!1-\pv{pos}!(\tikztotarget)!\pv{height}!270:(\tikztotarget)$)
			.. (\tikztotarget)\tikztonodes}},
	settings/.code={\tikzset{quiver/.cd,#1}
		\def\pv##1{\pgfkeysvalueof{/tikz/quiver/##1}}},
	quiver/.cd,pos/.initial=0.35,height/.initial=0}

\newtheoremstyle{introthms}
{}{}{\itshape}{}{\bfseries }{}{ }
{\thmname{#1} \thmnumber{#2}. \thmnote{\bfseries{(#3)}}}

\theoremstyle{introthms}
\newtheorem{introthm}{Theorem}

\newtheorem{introcor}[introthm]{Corollary}

\theoremstyle{plain}
\newtheorem*{un}{Theorem}

\theoremstyle{plain} 
\newtheorem{theorem}{Theorem}[section]

\newtheorem{lemma}[theorem]{Lemma}
\newtheorem{proposition}[theorem]{Proposition}
\newtheorem{corollary}[theorem]{Corollary}

\theoremstyle{definition}
\newtheorem{remark}[theorem]{Remark}

\newtheorem{definition}[theorem]{Definition}
\newtheorem{example}[theorem]{Example}
\newtheorem{construction}[theorem]{Construction}
\newtheorem{notation}[theorem]{Notation}

\newtheorem*{remark*}{Remark}

\title{Globalizing and stabilizing global $\infty$-categories}

\author{Sil Linskens}

\begin{document}
	
	\begin{abstract}
		We consider the question of cocompleting partially presentable parametrized $\infty$-categories in the sense of \cite{CLL_Partial}. As our main result we show that in certain cases one may compute such relative cocompletions via a very explicit formula given in terms of partially lax limits. We then apply this to equivariant homotopy theory, building on the work of \cite{CLL23} and \cite{CLL_Partial}, to conclude that the global $\infty$-category of globally equivariant spectra is the relative cocompletion of the global $\infty$-category of equivariant spectra. Evaluating at a group $G$ we obtain a description of the $\infty$-category of $G$-global spectra as a partially lax limit, extending the main result of \cite{LNP} for finite groups to $G$-global homotopy theory. Finally we investigate the question of stabilizing global $\infty$-categories by inverting the action of representation spheres, and deduce a second universal property for the global $\infty$-category of globally equivariant spectra, similar to that of \cite{LS23}.
	\end{abstract}
	
	\maketitle
	\setcounter{tocdepth}{1}
	\tableofcontents

	\section{Introduction}
	
	Global homotopy theory studies objects which have a ``compatible" action of all (compact Lie) groups within a designated family. For instance, consider ($G$-)equivariant $K$-theory, equivariant (stable) bordism, stable cohomotopy, and Borel cohomology. Each of these $G$-equivariant cohomology theories admits a definition which is in some sense uniform in the group $G$. As such each cohomology theory should, and in fact does, define a global stable homotopy type, that is an object of the $\infty$-category\footnote{We work in the context of higher category theory throughout, and so we will in what follows refer to $\infty$-categories simply as `categories.'} $\Sp_{\gl}$ of \textit{global spectra} in the sense of \cite{Schwede18}. In the unstable setting one studies $\Spc_{\gl}$, the category of \textit{global spaces} as originally defined by \cite{GH} (where they are called $\Orb$-spaces). Once again the objects of $\Spc_{\gl}$ should be spaces which are equipped with a collection of compatible actions. 
	
	However unstable global homotopy theory is more immediately described as the homotopy theory of spaces with singularities which locally look like the quotient of a space by a group action, i.e.~of orbispaces. In fact, the precise definition of neither global spaces nor global spectra is obviously an implementation of the initial motivation, that a global object should be a compatible family of equivariant objects. Nevertheless, in recent work Denis Nardin, Luca Pol and the author \cite{LNP} have proven a conjecture of Stefan Schwede, which shows that this heuristic does in fact lead to a precise definition. Namely the result shows that, informally, the data of a global space/spectrum is equivalent to the data of
	\begin{itemize}
		\item a $G$-space/spectrum $\mathrm{res}_G X$ for each group $G$ in the designated family,
		\item an $H$-equivariant map $f_\alpha\colon \alpha^*\mathrm{res}_G X\to \mathrm{res}_H X$ for each continuous group homomorphism $\alpha\colon H \to G$,
		\item a homotopy between the map $f_{c_g}$ induced by the conjugation isomorphism $c_g\colon G\rightarrow G$ and the map $l_g \colon c_g^* \mathrm{res}_G X \to \mathrm{res}_G X$ given by left multiplication by $g$,
		\item higher coherences for the homotopies,
	\end{itemize}
	which satisfy the following compatibility conditions:
	\begin{itemize}
		\item the maps $f_\alpha$ are functorial, so that $f_{\beta\circ \alpha}\simeq f_\beta\circ \beta^*(f_{\alpha})$ for all composable maps $\alpha$ and $\beta$, and $f_{\mathrm{id}}=\mathrm{id}$,
		\item the map $f_\alpha$ is an equivalence for every continuous \emph{injective} homomorphism $\alpha$.
	\end{itemize}
	
	The following theorem is the precise formulation of this.
	
	\begin{un}[\cite{LNP}*{Theorem 6.18, 11.10}]\label{thm:Sp_gl_lax_lim}
		There exist equivalences of symmetric monoidal categories
		\[
		\Spc_{\gl}\simeq \laxlimdag_{\Orb^{\op}\subset \Glo^{\op}} \Spc_\bullet \quad \text{and} \quad \Sp_{\gl} \simeq \laxlimdag_{\Orb^{\op}\subset \Glo^{\op}} \Sp_\bullet.
		\]
	\end{un}
	
	To unpack the notation slightly, we recall that $\Glo$ denotes the \textit{global indexing category} whose objects are given by a fixed family of compact Lie groups and whose morphisms are given by group homomorphisms up to conjugation, while $\Orb$ denotes the wide subcategory thereof spanned by the injective group homomorphisms. The notation $\laxlimdag F$ above refers to the \textit{partially lax limit} of a functor $F\colon \Glo^{\op}\rightarrow \Cat$. The base $\Glo$ is marked by the subcategory $\Orb$, and so the objects of the partially lax limit are lax \textit{away} from $\Orb$. We refer the reader to \Cref{def:partially_lax_limits} for a quick definition of partially lax limits, and \cite{LNP}*{Section 4} for an extended discussion. Most importantly for us, they provide a means for making precise the informal characterization of global spaces/spectra given before the theorem.
	
	While the previous theorem provides a useful description of the categories of global spectra and global spaces, it is slightly unmotivated: there are a priori many ways to make precise the data of a compatible collection of equivariant objects, and we may wonder why we should give priority to this choice? In this article we provide one answer by showing that, for the family of finite groups, the partially lax limit above comes about by universally solving a defect of the diagram $\Sp_\bullet$ of equivariant spectra. Explaining this defect is most convenient from the perspective of \textit{global categories}, which we recall now.
	
	In mathematics one often wants to study objects which come equipped with an action of a group. Typically these collections of $G$-equivariant objects assemble into a category, which becomes the main object of study. In the process of understanding such categories one crucially uses their functoriality in the group $G$\hspace{1pt}; i.e.~the fact that one can restrict actions along group homomorphisms. Thinking systematically about this functoriality leads to the definition of a \textit{global category}, which is roughly the data of 
	\begin{enumerate}
		\item a category $\CC_G$ for every finite group $G$\hspace{1pt};
		\item a restriction functor $\alpha^*\colon \CC(G) \to \CC(H)$ for every homomorphism $\alpha\colon H \to G$\hspace{1pt};
		\item higher structure which in particular witnesses that conjugate morphisms induce the same restriction functor.
	\end{enumerate}
	More precisely we define a global category to be a functor \[\CC\colon \Glo^{\op}\rightarrow \Cat,\quad \CB G\mapsto \CC_G.\] In the previous definition and for the remainder of this article, $\Glo$ will refer to the global indexing category for the family of finite groups. In particular $\Glo$ is equivalent to the $(2,1)$-category of finite connected groupoids.
	
	The study of global categories is the study of the abstract representation theory of finite groups, understood in a very broad sense. This study of course has a long history, but in this exact formalism was begun in \cite{CLL23}. The perspective taken there was that of parametrized/internal higher category theory, in the sense of \cite{exposeI} and \cite{MW21} respectively. This is a robust generalization of higher category theory, which comes with its own notions of adjunctions, colimits and so on. Applying these notions to global categories one recovers properties familiar from representation theory. For example a global category $\CC$ admitting certain parametrized colimits, which we call \textit{equivariant colimits}, if each category $\CC_G$ admits colimits preserved by restriction, and there exist induction functors $\mathrm{ind}_H^G\colon \CC_H\mapsto \CC_G$ left adjoint to the restriction $i^*\colon \CC_G\rightarrow \CC_H$ along an inclusion $i\colon H \rightarrow G$, which furthermore satisfy an analogue of the classical double coset formula. Having all parametrized colimits implies the further existence of a quotient functor $(-)/N\colon \CC_G\rightarrow \CC_{G/N}$ which is left adjoint to the inflation functor $p^*\colon \CC_{G/N}\rightarrow \CC_G$ for every surjective homomorphism $p\colon G\rightarrow G/N$ of finite groups.
	
	With this background we can again consider the diagram
	\[\Sp_\bullet \colon \Glo^{\op}\rightarrow \Cat, \quad G \mapsto \Sp_G\]
	sending the group $G$ to $\Sp_G$, the category of genuine $G$-spectra. The crucial observation is that as a global category it suffers from one defect: it does not admit all parametrized colimits. This is a consequence of the fact that the restriction functor $\alpha^*\colon \Sp_G\rightarrow \Sp_H$ does not admit a left adjoint when $\alpha$ is a \textit{non-injective} group homomorphism, see \Cref{ex:eq_spectra}. In other words, it is not possible to construct a quotient functor $(-)/N\colon \Sp_G\rightarrow \Sp_{G/N}$ left adjoint to inflation. Nevertheless it does admit equivariant colimits, in the sense of the previous paragraph.
	
	In this paper we will show that one can freely add the missing parametrized colimits to a nice global category which admits equivariant colimits, and moreover that the value of the resulting global category at a group $G$ admits an explicit formula in terms of partially lax limits. Motivated by examples, we call this process \emph{globalization}. From this it follows that $\Sp_\gl$ naturally arises from the process of freely adding the missing parametrized colimits to the global category $\Sp_\bullet$. In fact we will obtain a stronger result. The category $\Sp_{\gl}$ admits a canonical parametrized enhancement 
	\[\Spbulgl \colon \Glo^{\op}\rightarrow \Cat,\quad \CB G \mapsto \Sp_{G\text{-}\gl}\] given by sending the groupoid $\CB G$ to the category of $G$-global spectra, as defined by \cite{LenzGglobal}. We call this the global category of \emph{globally equivariant spectra}. As the main result of this paper we will prove that this is equivalent to the globalization of $\Sp_\bullet$. To state the theorems we require some notation; careful definitions will be given in the body of the paper.
	
	\begin{definition}
		A global category $\CC$ is called \textit{equivariantly presentable} if $\CC\colon \Glo^{\op}\rightarrow \Cat$ factors through the subcategory $\PrLun$, the functor $\mathrm{res}^G_H \colon \CC_G \rightarrow \CC_H$ admits a left adjoint $\mathrm{ind}_H^G\colon \CC_H\rightarrow \CC_G$ for every subgroup inclusion $H\subset G$, and these left adjoints satisfy a categorified double coset formula, also known as the Beck--Chevalley condition. Similarly $\CC$ is called \textit{globally presentable} if, moreover, the restriction functors along surjective group homomorphisms also admits left adjoints which also satisfy the Beck--Chevalley condition.
		
		We define $\PrL{\Glo}{\Orb}$ and $\PrL{\Glo}{\L}$ to be the categories of equivariantly presentable and globally presentable global categories respectively.
	\end{definition}
	
	\begin{definition}
		Let $\CC$ be a global category, then we define a functor 
		\[\Glob(\CC)\colon \Glo^{\op}\rightarrow \Cat,\] 
		called the \textit{globalization} of $\CC$, via the assignment 
		\[\Glob(\CC)_G = \laxlimdag((\Glo_{/G})^{\op}\rightarrow \Glo^{\op}\xrightarrow{\CC} \Cat),\]
		where an edge is marked in $\Glo_{/G}^{\op}$ is marked if its projection to $\Glo^{\op}$ is a faithful functor of groupoids, i.e.~lands in $\Orb^{\op}$. The functoriality of this assignment in $\Glo^{\op}$ is induced by the contravariant functoriality of partially lax limits applied to the pushforward functoriality of the slices ${\Glo}_{/G}$. 
	\end{definition}

	\begin{introthm}\label{intro_thm_A}
		Suppose $\CC$ is an equivariantly presentable global category. Then $\Glob(\CC)$ is a globally presentable global category. Furthermore the restriction \[\Glob\colon \PrL{\Glo}{\Orb} \rightarrow\PrL{\Glo}{\L}\] is left adjoint to the (non-full) inclusion $\PrL{\Glo}{\L}\rightarrow \PrL{\Glo}{\Orb}$.
	\end{introthm}
	
	See \Cref{thm:Free_colimit} for a more precise statement, which is stated in the generality of a suitable pair $(T,S)$ of a category $T$ and a subcategory $S\subset T$. More precisely we require $T$ to be an orbital category and $S$ to be an orbital subcategory of $T$. To prove this result we require a long list of results about partially lax limits, which we collect in \Cref{sec:appendix_A}. For example we give sufficient conditions for the existence of (co)limits in a partially lax limit of categories, and give two criteria via which one obtains adjunctions between two partially lax limits.
	
	We can now apply this to the global category $\Sp_\bullet$. Using the main results of previous joint work \cite{CLL23} and \cite{CLL_Partial} of Bastiaan Cnossen, Tobias Lenz and the author we prove:
	
	\begin{introthm}\label{intro_thm_B}
		There exists an equivalence 
		\[\Spbulgl \simeq \Glob(\Sp_\bullet)\] 
		of global categories.
	\end{introthm}
	
	This theorem is however much more then just a consistency check. For example it has the following significant non-parametrized consequence.
	
	\begin{introcor}\label{intro_thm_C}
		Let $G$ be a finite group. There exists an equivalence
		\[
		\Sp_{G\text{-}\gl} \simeq \laxlimdag_{(\Glo_{/G})^{\op}} \Sp_\bullet.
		\] 
		
	\end{introcor}
	
	For the family of finite groups, this generalizes the main result of \cite{LNP} from global to $G$-global homotopy theory. We also emphasize that this result is independent of \cite{LNP}, and so gives a new proof even when $G$ is the trivial group. 
	
	By previous work \cite{CLL23}, the global category $\Spbulgl$ has the advantage that it admits a universal property: it is the free globally presentable equivariantly stable global category on a point. Informally, an equivariantly presentable global category is equivariantly stable if each category $\CC_G$ is stable, and the functors $\mathrm{ind}^G_H$ are also right adjoint to restriction. This universal property is in fact the crucial ingredient for the previous result. By \cite{CLL_Partial}, $\Sp_\bullet$ is itself the free equivariantly presentable equivariantly stable global category on a point. Therefore the theorem above follows from the fact that globalization preserves equivariant stability, as we show in \Cref{prop:rel_cocomp_semi}.
	
	\subsection{Representation stability}

	Having recognized the global category $\Spbulgl$ of globally equivariant spectra as the globalization of the global category $\Sp_\bullet$ of equivariant spectra, we can immediately deduce universal properties for the former from universal properties of the latter. One such universal property is very close to the definition of genuine equivariant spectra: $\Sp_G$ is given by inverting the representations spheres in $\Spc_{G,\ast}$, the category of pointed $G$-spaces. To discuss this systematically, we recall that an equivariantly presentable global category $\CC$ is pointed if $\CC_G$ is pointed for all $\CB G\in \Glo$. In this case one can construct a tensoring of $\CC$ by $\Spc_{\bullet,\ast}$, the global category of pointed equivariant spaces. Given this we can formulate the following definition.
	
	\begin{definition}
		We say a pointed equivariantly presentable global category $\CC$ is \emph{$\Rep$-stable} if for every $G\in \Glo$ and every $G$-representation $V$ the functor 
		\[S^V\otimes - \colon \CC_G\rightarrow \CC_G\] 
		is an equivalence. We write $\PrLVst{\Glo}{\Orb}$ for the full subcategory of $\PrL{\Glo}{\Orb}$ spanned by the $\Rep$-stable equivariantly presentable global categories.
	\end{definition}
	
	Now consider an arbitrary pointed equivariantly presentable global category $\CC$. As we make precise in \Cref{def:Stab_Orb}, inverting the action of the representation spheres pointwise defines a new global category, which we denote by $\Stab^{\Orb}(\CC)$. 
	
	One can show that $\Stab^{\Orb}(\CC)$ is a $\Rep$-stable equivariantly presentable global category, and furthermore that the functor 
	\[
	\Stab^{\Orb}\colon \PrL{\Glo}{\Orb}\rightarrow \PrLVst{\Glo}{\Orb}
	\] 
	defines a left adjoint to the inclusion of $\Rep$-stable equivariantly presentable global categories into $\PrL{\Glo}{\Orb}$. In particular we conclude that $\Sp_\bullet$ is the free $\Rep$-stable equivariantly presentable global category generated by a point.
	
	We note that the fact that $\Stab^{\Orb}(\CC)$ is again equivariantly presentable is not a formality. To emphasize this we observe that while $\Spc_{\bullet,\ast}$ is globally presentable, $\Sp_\bullet \simeq \Stab^{\Orb}(\Spc_{\bullet,\ast})$ is not. So the process of stabilizing pointwise can in general destroy the existence of certain parametrized colimits. This makes the $\Rep$-stabilizations of \emph{globally} presentable global categories much more complicated in general. We can nevertheless prove the following theorem.
	
	\begin{introthm}\label{intro_thm_D}
		The globalization of a $\Rep$-stable global $\infty$-category is again $\Rep$-stable. In particular $\Spbulgl$ is the free $\Rep$-stable globally presentable global category on a single generator. 
	\end{introthm}
	
	Such a universal property for global spectra was first suggested by David Gepner and Thomas Nikolaus \cite{OWF}. In the setting of \emph{global model categories}, a similar universal property was proved in \cite{LS23}.
	
	Finally, note that a partially lax limit of symmetric monoidal categories is canonically symmetric monoidal. Therefore $\Spbulgl$ is canonically a symmetric monoidal global category. We also prove a symmetric monoidal analogue of the previous theorem.
	
	\begin{introcor}\label{intro_thm_E}
		The global $\infty$-category $\Spbulgl$ of globally equivariant spectra is the initial $\Rep$-stable globally presentable symmetric monoidal global category.
	\end{introcor}

	\subsection*{Organization}
	
	In \Cref{sec:S-pres} we introduce the notion of $S$-presentable $T$-categories, for a pair of an orbital category $T$ and an orbital subcategory $S$ of $T$. In \Cref{sec:appendix_A} we collect facts about partially lax limits which are used in \Cref{sec:Glob} to deduce \Cref{intro_thm_A}. In \Cref{sec:Glob} we explain how to construct the free cocompletion with relations $\Free(\CC)$ of a $S$-presentable $T$-category $\CC$ using partially lax limits, leading to a proof of \Cref{intro_thm_A}. Specializing to the pair $(\Glo,\Orb)$ we obtain the globalization functor $\Glob(-)$. We then recall in \Cref{sec:P-stable} the relevant background on parametrized semiadditivity and stability, and show that in certain cases $\Free(-)$ preserves this property. We then apply this to the global context to conclude \Cref{intro_thm_B} and \Cref{intro_thm_C}. Finally in \Cref{sec:Rep-stable} we introduce the notion of $\Rep$-stable equivariantly presentable global categories and explain how to construct $\Rep$-stabilizations via inverting the action of representation spheres. We then consider the interaction of globalization with $\Rep$-stabilization to conclude \Cref{intro_thm_D} and \Cref{intro_thm_E}.
	
	\subsection*{Acknowledgements}
	We would like to thank Bastiaan Cnossen and Tobias Lenz for the fruitful collaboration which was crucial for the proof of Theorem B. We also thank them for reading an earlier draft of this article. Finally we would like to thank Miguel Barrero, Bastiaan Cnossen, David Gepner, Fabian Hebestreit, Marin Jannsen, Tobias Lenz, Denis Nardin, Thomas Nikolaus, Luca Pol, Stefan Schwede and Sebastian Wolf for helpful discussions. The author is an associate member of the Hausdorff Center for Mathematics at the University of Bonn, supported by the DFG Schwerpunktprogramm 1786 ``Homotopy Theory and Algebraic Geometry'' (project ID SCHW 860/1-1).
	
	\section{Partial presentability}\label{sec:S-pres}
	
	In this section we recall some basic definitions from parametrized category theory, and introduce the notion of partial presentability. This section is rather terse, and the reader may benefit from consulting \cite{CLL_Partial}*{Section 2}.
	
	\begin{notation}
		Given a category $T$ we define $\bbF_T$, the \emph{finite coproduct completion} of $T$, as the smallest full subcategory of $\PSh(T)$ which is closed under coproducts and contains the image of the Yoneda embedding.
	\end{notation}
	
	\begin{definition}
		We define $\Cat_T$, the category of \emph{$T$-categories}, to equal $\Fun^\times(\bbF_T^{\op},\Cat)$, the $\infty$-category of finite product preserving functors from $\bbF_T^{\op}$ to $\Cat$.
	\end{definition}
	
	\begin{remark}
		The objects of $\Cat_T$ are referred to as $T$-categories, and morphisms in $\Cat_T$ are called $T$-functors. Note that restriction along the inclusion $T^{\op}\subset \bbF_T^{\op}$ induces an equivalence
		\[\Cat_T \xrightarrow{\sim} \Fun(T^{\op},\Cat).\]
	\end{remark}
	
	\begin{notation}
		Given a $T$-category $\CC$ we will typically denote $\CC(X)$ by $\CC_X$ and $\CC(f)$ by $f^*$. If $f^*$ has a left or right adjoint, we will denote it by $f_!$ and $f_*$ respectively. We will write the component of a $T$-functor $F\colon \CC \rightarrow \CD$ at $X$ by $F_X\colon \CC_X\rightarrow \CD_X$.
	\end{notation}
	
	\begin{remark}
		$\Cat_T$ admits an enhancement to a $T$-parametrized category $\ul{\Cat}_T$ via the assignment $\ul{\Cat}_T(X) \coloneqq \Cat_{T_{/X}}$.
	\end{remark}
	
	We fix a wide subcategory inclusion $S\subset T$. Note that the inclusion $S \subset T$ induces a functor  $\bbF_S\rightarrow \bbF_T$, which exhibits $\bbF_S$ as a wide subcategory of $\bbF_T$.
	
	\begin{definition}
		We say $S\subset T$	is an \textit{orbital subcategory} if the pullback of a morphism in $\bbF_S$ along any morphism in $\bbF_T$ exists in $\bbF_T$ and is again in $\bbF_S$. We say $T$ is \textit{orbital} if it is an orbital subcategory of itself. We say $(T,S)$ is an \textit{orbital pair} if $T$ is orbital and $S$ is an orbital subcategory of $T$.
	\end{definition}
	
	\begin{example}
		We define $\Glo$ to be the $(2,1)$-category of finite connected groupoids $\bB G$ and $\Orb$ the subcategory spanned by the faithful functors. We claim that $(\Glo,\Orb)$ is an orbital pair. Observe that $\bbF_\Glo$ is equivalent to the $(2,1)$-category of finite groupoids, which admits all homotopy pullbacks. The subcategory $\bbF_\Orb$ is the wide subcategory on the faithful maps of groupoids, and thus the orbitality of $\Orb$ is equivalent to the observation that pullbacks of faithful maps of groupoids are again faithful.
	\end{example}
	
	\begin{example}
		The orbit category $\Orb_G$ of a finite group $G$ is orbital.
	\end{example}
	
	\begin{example}
		Suppose $(T,S)$ is an orbital pair. Then $(T_{/X}, \pi_X^{-1}(S))$ is again an orbital pair, where $\pi_X^{-1}(S)$ is the preimage of $S$ under the functor $\pi_X\colon T_{/X}\rightarrow T$.
	\end{example}
	
	Before we state the definition of $S$-presentability we recall the following categorical notion:
	\begin{definition}
		Consider a commutative square
		\[\begin{tikzcd}
			\CC & {\CC'} \\
			\CD & {\CD'}
			\arrow["{F'}"', from=1-1, to=2-1]
			\arrow["F", from=1-2, to=2-2]
			\arrow["{G'}", from=2-1, to=2-2]
			\arrow["G", from=1-1, to=1-2]
		\end{tikzcd}\] in $\Cat$ such that both $F$ and $F'$ are right adjoints, with left adjoints $L$ and $L'$ respectively. We say such a square is \textit{left adjointable} if the Beck-Chevalley transformation
		\[
		L'G' \xRightarrow{\;\eta\;\;} L'G'FL \xRightarrow{\;\sim\;\;} L'F'GL \xRightarrow{\;\epsilon\;\;} GL
		\]
		is an equivalence. If $F$ and $F'$ are instead left adjoints then we can dually define the notion of right adjointability.
	\end{definition}
	
	We now introduce the notion of $S$-presentability for $T$-categories. 
	
	\begin{definition}\label{def:S-presentable}
		Let $(T,S)$ be an orbital pair. We say a $T$-category $\CC$ is \emph{$S$-presentable} if 
		\begin{enumerate}
			\item $\CC$ is fiberwise presentable, i.e.~lifts to a functor $\CC\colon \bbF_T^{\op}\rightarrow \PrLun$;
			\item $p^*$ has a left adjoint for all $p\in \bbF_S$;
			\item For every pullback square
			\[\begin{tikzcd}
				{X'} & {X} \\
				{Y'} & {Y}
				\arrow["g'", from=1-1, to=1-2]
				\arrow["p'"', from=1-1, to=2-1]
				\arrow["g", from=2-1, to=2-2]
				\arrow["p", from=1-2, to=2-2]
				\arrow["\lrcorner"{anchor=center, pos=0.125}, draw=none, from=1-1, to=2-2]
			\end{tikzcd}\]
			in $\bbF_T$ such that $p$ (and therefore $p'$) is in $\bbF_S$ the square 
			\[\begin{tikzcd}
				{\CC_{Y}} & {\CC_{Y'}} \\
				{\CC_{X}} & {\CC_{X'}}
				\arrow["{g^*}", from=1-1, to=1-2]
				\arrow["{(p)^*}"', from=1-1, to=2-1]
				\arrow["{(g')^*}", from=2-1, to=2-2]
				\arrow["{(p')^*}", from=1-2, to=2-2]
			\end{tikzcd}\]
			is left adjointable. We may refer to this condition by saying that $\CC$ satisfies base-change for morphisms in $\bbF_S$.
		\end{enumerate}
		We say a functor $F\colon \CC\rightarrow \CD$ between $S$-presentable categories is \emph{$S$-cocontinuous} if for all $X\in \mathbb{F}_T$ the functor $F_X$ admits a right adjoint and the square 
		\[\begin{tikzcd}
			{\CC_Y} & {\CD_Y} \\
			{\CC_X} & {\CD_X}
			\arrow["{F_X}", from=2-1, to=2-2]
			\arrow["{F_Y}", from=1-1, to=1-2]
			\arrow["{p^*}"', from=1-1, to=2-1]
			\arrow["{p^*}", from=1-2, to=2-2]
		\end{tikzcd}\]
		is left adjointable for all $p\colon X\rightarrow Y$ in $\mathbb{F}_S$. 
		
		We define the category of $S$-presentable $T$-categories $\PrL{T}{S}$ as the subcategory of $\Cat_T$ spanned by the $S$-presentable $T$-categories and $S$-cocontinuous functors. One can show that the assignment $\uPrL{T}{S}(X) \simeq \PrL{T_{/X}}{\pi_X^{-1}(S)}$  is a parametrized subcategory of $\ul{\Cat}_T$.
	\end{definition}
	
	\begin{remark}
		The notion of $S$-presentability has been previously introduced by \cite{CLL_Partial} in the generality of a cleft category $S\subset T$ \cite{CLL_Partial}*{Definition 3.2}. In certain ways cleft categories are less general than an orbital pair, but in others ways they are much more general. For example in a cleft category we only require that pullbacks of maps in $S$ along maps in $T$ land in the image of $\PSh(S)$ in $\PSh(T)$, instead of in $\bbF_S$. We expect that the results of this section are true in a generality which encompasses cleft categories, but have been unable to show this so far.
		
		We note that for a $S\subset T$ which is both a cleft category and an orbital pair, a $T$-category $\CC$ is $S$-presentable in our sense if and only if it is $S$-presentable in the sense of \cite{CLL_Partial}*{Definition 4.3}.
	\end{remark}
	
	\begin{remark}\label{rem:PrL_the_same}
		$\PrL{T}{T}$ is equivalent to the category $\PrL{T}{\L}$ of presentable categories internal to the presheaf topos $\PSh(T)$ in the sense of \cite{MW22}, see Theorem A of \textit{loc.~cit.} Therefore we will denote $\PrL{T}{T}$ by $\PrL{T}{\L}$. By the parametrized adjoint functor theorem \cite{MW22}*{Proposition 6.3.1} a $T$-cocontinuous functor between $T$-presentable categories is equivalent to a parametrzied left adjoint, which we may define as an adjunction in the $2$-category $\Fun(T^{\op},\Cat)$.
	\end{remark}
	
	\begin{remark}\label{rem:internal_char_of_colims}
		Let $f\colon X\rightarrow Y$ be a map in $\bbF_T$, and consider the adjunction \[f_!\colon {\bbF_{T}}_{/X}\rightleftarrows {\bbF_T}_{/Y}\cocolon f^*.\] Note that both functors are coproduct preserving, and so induce an adjunction 
		\[
		f^*\colon \Cat_{T_{/Y}} \rightleftarrows \Cat_{T_{/X}}\cocolon f_*,
		\] 
		where $f^*$ and $f_*$ are given by precomposing by $f_!$ and $f^*$ respectively. By \cite{CLL23}*{Lemma 2.3.14} Conditions (2) and (3) of the previous definition together are equivalent to the claim that for all $p\colon X\rightarrow Y$ in $\bbF_S$, the unit functor 
		\[
		\pi_Y^*\CC \xrightarrow{\ul{p}^*} p_*p^*\pi_Y^*\CC
		\]
		of $T_{/Y}$-categories admits a parametrized left adjoint, which we will denote by $\ul{p}_!$.
	\end{remark}
	
	\begin{example}
		Applying the previous definitions to the orbital pair $(\Glo,\Orb)$ we recover the notion of equivariant presentability from the introduction, first defined in \cite{CLL_Partial}. Applied to $(\Glo,\Glo)$ we obtain the notion of global presentability.
	\end{example}
	
	\begin{example}\label{ex:S-spaces}
		We define the $T$-category $\Spc_\bullet^T \coloneqq \PSh(T)_{/\bullet}$ of \emph{$T$-spaces}, where $\Spc_\bullet^T$ is functorial in pullback. We define $\Spc^S_\bullet$ as the full $T$-subcategory of $\Spc_\bullet^T$ which at $X\in T$ is given by the smallest full category closed under colimits and containing those maps $Z\rightarrow X$ which are in $S$. Because $S$ is orbital this forms a parametrized subcategory. By \cite{MW21}*{Remark 7.3.4} this is the free $S$-presentable $T$-category on a point: for every $S$-presentable $T$-$\infty$-category there exists an equivalence
		\[
		\Hom_{\PrL{T}{S}}(\Spc_\bullet^S,\CC) \simeq \core(\Gamma\CC),
		\] where $\Gamma\CC \coloneqq \lim_{T^{\op}}\CC$ and $\core(\Gamma\CC)$ is the subcategory of $\Gamma\CC$ spanned by the equivalences. Applied to $S=T$ we find that $\Spc_\bullet^T$ is the free $T$-presentable $T$-category on a point.
	\end{example}
	
	\begin{example}\label{ex:eq_spaces}
		In the special case of $(T,S) = (\Glo,\Orb)$, $\Spc^\Orb_\bullet$ is equivalent to the global category $\Spc_\bullet$, which sends $\CB G$ to the category of $G$-spaces. See \cite{CLL_Partial}*{Theorem 5.3} for a proof of this fact. Similarly $\Spc^\Glo_\bullet$ is equivalent to the category of globally equivariant spaces, which sends $\CB G$ to the category of $G$-global spaces in the sense of \cite{LenzGglobal}, see \cite{CLL23}*{Theorem 3.2.2}.
	\end{example}
	
	Our key motivation for introducing the notion of $S$-presentability is to have a convenient formalism to engage with the following example.
	
	\begin{example}\label{ex:eq_spectra}
		We define the global category of (genuine) equivariant spectra $\Sp_\bullet$ by sending $\CB G$ to the category of \textit{$G$-spectra}. Formally one may define it as the initial functor $\CC\colon \Glo^{\op}\rightarrow \CAlg(\PrL{}{\mathrm{L}})$ under the functor
		\[
		\Spc_{\bullet,\ast}\colon \Glo^{\op}\rightarrow \CAlg(\PrL{}{\mathrm{L}}), \quad \CB G \mapsto \Spc_{G,\ast},
		\] such that the representation spheres are pointwise invertible. Such a functor exists by results of \cite{Ro15}; we refer the reader to \Cref{sec:Rep-stable} for more details.
		
		We can compare this to more classical definitions. Namely, because the representation spheres are invertible in the category of genuine $G$-spectra, we obtain, by the universal property of $\CC$, a comparison natural transformation from $\CC$ to the global category of equivariant spectra $\Sp_\bullet$, given by applying Dwyer--Kan localization pointwise to the diagram of relative categories sending $\CB G$ to orthogonal $G$-spectra together with the stable equivalences. The latter is the definition of the global category of equivariant spectra given in \cite{CLL_Partial}*{Section 9.1}. The resulting natural transformation is pointwise an equivalence by the results of \cite{gepner2020equivariant}*{Appendix C}, and so we conclude that our definition agrees with the usual definition of genuine equivariant spectra.
		
		To connect to the discussion of partial presentability, we observe that $\Sp_\bullet$ is equivariantly presentable. While this is nothing more then a collection of classical statements about equivariant spectra which are surely well-known to experts, it is also a special case of \Cref{thm:Stab_P}. However $\Sp_\bullet$ is \textit{not} globally presentable: the restriction functor $q^*\colon \Sp_{G}\rightarrow \Sp_{H}$ does not have a left adjoint whenever $q\colon H\rightarrow G$ is a non-injective group homomorphism. The existence of such a left adjoint is obstructed by the tom Dieck splitting, which implies that $q_*$ does not preserve compact objects when $q$ is not injective. By \cite{BDS15}*{Theorem 3.3} this implies that $q^*$ cannot preserve all limits.
	\end{example}
	
	\section{Partially lax limits}\label{sec:appendix_A}
	
	One of the main goals of this article is to give a construction of the relative cocompletion of $S$-presentable $T$-categories using partially lax limits. In this section we will recall the definition of partially lax limits of categories, originally due to \cite{Berman} in the higher categorical context. Then we will collect a variety of facts about them which we require to provide a formula for relative cocompletion. For example in \Cref{subsec:lim_in_part} we consider the question of when partially lax limits admit limits and colimits, and how they are computed. Then in \Cref{subsec:adj_of_part} we give two methods for constructing adjunctions between partially lax limits.
	
	\begin{definition}\label{def:marked}
		A marked category $(\CI,\CW)$ consists of a category $\CI$ equipped with a replete subcategory $\CW$. We write $\Cat^\dagger$ for the category of marked categories and marking preserving functors.
	\end{definition}
	
	\begin{definition}
		Let $(\CI,\CW)$ be a marked category and let $F\colon \CI\rightarrow \Cat$ be a functor. Then we view the cocartesian unstraightening $\Unco{F}$ canonically as a marked category by marking all of the cocartesian morphisms which live over morphisms in $\CW$. 
	\end{definition}
	
	\begin{definition}
		Given two marked categories $\CC$ and $\CD$ we write $\Fun^\dagger(\CC,\CD)$ for the full subcategory of $\Fun(\CC,\CD)$ spanned by those functors which preserve marked morphisms. Suppose $\CC$ and $\CD$ both admit a functor $F$ and $G$ respectively to a category $\CI$. Then we define $\Fun_{\CI}(\CC,\CD)$ to be the pullback 
		\[\begin{tikzcd}
			{\Fun_{\CI}^\dagger(\CC,\CD)} & {\Fun^\dagger(\CC,\CD)} \\
			{\{F\}} & {\Fun(\CC,\CI)}
			\arrow["{G_*}", from=1-2, to=2-2]
			\arrow[""{name=0, anchor=center, inner sep=0}, from=2-1, to=2-2]
			\arrow[from=1-1, to=2-1]
			\arrow[from=1-1, to=1-2]
			\arrow["\lrcorner"{anchor=center, pos=0.125}, draw=none, from=1-1, to=0]
		\end{tikzcd}\]
	\end{definition}
	\begin{definition}\label{def:partially_lax_limits}
		Given a marked category $(\CI,\CW)$ and a functor $F\colon \CI\rightarrow \Cat$, we define the \textit{partially lax limit of $F$ with respect to $\CW$} \[\laxlimdag_{(\CI,\CW)} F\coloneqq \Fun_{\CI}^{\dagger}(\CI,\Unco{F})\] as the category of sections of the cocartesian unstraightening $\Unco{F}\to \CI$ of $F$ which preserve marked edges, i.e.~send morphisms in $\CW$ to cocartesian edges of $\Unco{F}$.
	\end{definition}
	
	We will sometimes drop the reference to the marking on $\CI$ when it is either implicit or clear from context.
	
	\begin{remark}\label{rem:sections_unwind}
		Consider a section $s\colon \CI\rightarrow \Unco{F}$. Note that for every $i\in \CI$, $s(i)$ lives in the fiber of $\Unco{F}$ over $i$ and so may view $X_i \coloneqq s(i)$ as an object of $F(i)$. Next we may consider the map $s(\alpha)\colon X_i\rightarrow X_{i'}$ associated to a morphism $\alpha\colon i \to i'$ in $\CI$. Once again because $s$ is a section, $s(\alpha)$ lives over $\alpha$. Factoring $s(\alpha)$ into a cocartesian edge followed by a map living in the fiber over $i'$ gives a morphism $s_\alpha\colon F(\alpha)X_i\rightarrow X_{i'}$. Note that $s$ is an object of the partially lax limit with respect to $\CW$ if and only if $s_\alpha$ is an equivalence for all edges $\alpha\in \CW$. The remaining data contained in the section $s$ encodes compatibility and coherence data for the collection of morphisms $s_\alpha$.
	\end{remark}
	
	First we state two simple results, which together will imply that the relative cocompletion preserves $T$-categories.
	
	\begin{proposition}\label{prop:lim_of_lim_ext}
		Consider a diagram $F\colon \CI\rightarrow \Cat$ and write $\CI^\Pi$ for the finite product completion of $\CI$ and $\tilde{F}\colon \CI^\Pi\rightarrow \Cat$ for the extension of $F$ to $\CI^\Pi$. Then the canonical map
		\[
		\laxlimdag_{\CW^\Pi\subset \CI^\Pi} \tilde{F} \rightarrow \laxlimdag_{\CW \subset \CI} F
		\] 
		is an equivalence.
	\end{proposition}
	
	\begin{proof}
		A simple computation shows that relative right Kan extension, in the sense of \cite{HTT}*{Definition 4.3.2.2}, provides an inverse.
	\end{proof}
	
	Suppose $\CJ_{\bullet}\colon \CI\rightarrow \Cat^\dagger$ is a diagram in marked categories, and that $\CI$ is itself marked. Then we will canonically consider the unstraightening $\Unco{\CJ_\bullet}$ as a marked category by marking both the cocartesian edges in $\Unco{\CJ_\bullet}$ which lie over marked edges of $\CI$, as well as the marked edges in each fiber. Recall that there exists a functor $\Unco{\CJ_\bullet}\rightarrow \colim \CJ_\bullet$ which exhibits the target as a localization of the source at the cocartesian edges.
	
	\begin{proposition}\label{prop:lim_on_colim_diagram}
		Consider a diagram $\CJ_{-}\colon \CI\rightarrow \Cat^\dagger$ in marked categories and a cocone $\{F_i\colon\CJ_i\rightarrow \Cat\}_{i\in \CI}$, which induces a functor $F\colon \Unco{\CJ_\bullet}\rightarrow \colim \CJ_i\rightarrow \Cat$. Then there exists an equivalence
		\[
		\laxlimdag_{\Unco{\CJ_\bullet}} F \simeq \laxlimdag_\CI \laxlimdag_{\CJ_i} F_i
		\] is an equivalence. 
	\end{proposition}
	
	\begin{proof}
		This follows from the following chain of equivalences:
		\begin{align*}
			\laxlimdag_{\Unco{\CJ_\bullet}} F &\coloneqq \Fun_{\Unco{\CJ_\bullet}}^{\CW-\textup{cocart}}(\Unco{\CJ_\bullet}, \Unco{F}) \\
			&\simeq \laxlimdag_\CI \Fun_{\Unco{\CJ_\bullet}}^{\CW_i-\textup{cocart}}(\CJ_i, \Unco{F}) \\
			&\simeq \laxlimdag_\CI\Fun_{\CJ_i}^{\CW_i-\textup{cocart}}(\CJ_i,\Unco{F_i})
			\eqqcolon  \laxlimdag_\CI \laxlimdag_{\CJ_i} F_i,
		\end{align*}
		where the first equivalence is justified by \cite{LNP}*{Proposition 4.15}.
	\end{proof}
	
	\begin{remark}\label{rem:lim_on_colim_diagram}
	Suppose that a functor $G\colon \CI\rightarrow \Cat$ sends a collection of edges $\CV$ to equivalences, and that $\CI$ is marked by a collection of edges $\CW$. Writing $\tilde{G}\colon \CI[\CV^{-1}]\rightarrow \Cat$ for the functor induced by $G$, we obtain an equivalence
	\[
	\laxlimdag_{\CI} G \coloneqq \Fun_{\CI}^{\CW-\textup{cocart}}(\CI,\Unco{G}) \simeq \Fun_{\CI[\CV^{-1}]}^{\CW'-\textup{cocart}}(\CI[\CV^{-1}],\Unco{\tilde{G}}) \eqqcolon \laxlimdag_{\CI[\CV^{-1}]} \tilde{G},
	\] 
	where $\CW'$ is the image of $\CW$ in $\CI[\CV^{-1}]$. The functor $F$ from the previous proposition is by definition of this form, and so in the left-hand side of the equivalence stated by the we may pass to the localization of $\Unco{F}$ at the cocartesian edges which lie over marked edges in $\CI$. In the extreme case that every edge of $\CI$ is marked, we obtain an equivalence
	\[
	\laxlimdag_{\colim \CJ_i} F \rightarrow \lim_\CI \laxlimdag_{\CJ_i} F_i.
	\]
	\end{remark}
	
	\subsection{Limits and colimits in partially lax limits}\label{subsec:lim_in_part}
	In this subsection we give two propositions which respectively provide sufficient conditions for the existence of limits and colimits in partially lax limits. We begin by considering the case of fully lax limits.
	
	\begin{proposition}\label{prop:laxlim_limits}
		Let $F\colon \CI\rightarrow \Cat$ be a functor such that each category $F(i)$ admits limits of shape $\CJ$ for all $i\in \CI$. Then $\laxlim F$ admits limits of shape $\CJ$, and a section $s\colon \CI\rightarrow \Unco{F}$ living over a $\CJ$-shaped diagram $\{s_j\}_{j\in\CJ}$ is a limit if and only if $s(i) \simeq \lim_\CJ s_j(i)$ for all $i\in \CI$.
	\end{proposition}
	
	\begin{proof}
		Recall that $\laxlim_{\CI} F \coloneqq \Fun_\CI(\CI,\Unco{F})$. Therefore this is the dual of \cite{HTT}*{Proposition 5.1.2.2}.
	\end{proof}
	
	Recall that we write $\Cat^{L}$ for the wide subcategory of $\Cat$ spanned by the left adjoint functors.
	
	\begin{proposition}
		Let $F\colon \CI\rightarrow \Cat^L$ be a functor such that each category $F(i)$ admits colimits of shape $\CJ$ for all $i\in\CI$. Then $\laxlim F$ admits colimits of shape $\CJ$ and a section $s\colon \CI\rightarrow \Unco{F}$ living under a $\CJ$-shaped diagram $\{s_j\}_{j\in\CJ}$ is a colimit if and only if $s(i) \simeq \colim s_j(i)$ for all $i\in \CI$.
	\end{proposition}
	
	\begin{proof}
		Because $F(f)$ is a left adjoint for all $f\colon i\rightarrow i'$ in $\CI$, $\Unco{F}$ is also a cartesian fibration by \cite{HTT}*{Corollary 5.2.2.5}. Therefore the result follows from \cite{HTT}*{Proposition 5.1.2.2}.
	\end{proof}
	
	Suppose $\CI$ is a marked category. We will now give a criteria for the inclusion $\laxlimdag F\subset \laxlim F$ to preserve limits and colimits. We begin with some preparation.
	
	\begin{notation}
		Consider a cocartesian fibration $\Unco{F}\rightarrow \CI$. Given a morphism $f\colon x_i\rightarrow x_j$ in $X$ which lives over the morphism $\alpha\colon i\rightarrow j$ in $\CI$, we write $f_\alpha\colon F(\alpha) x_i \rightarrow x_j$ for the morphism obtained by factoring $f$ into a cocartesian followed by a fiberwise edge.
	\end{notation}
	
	\begin{lemma}\label{lem:pushforward_composite}
		Let $\Unco{F}\rightarrow \CI$ be a cocartesian fibration. Consider a pair of composable morphisms 
		\[x_i\xrightarrow{f} x_j\xrightarrow{g} x_k\] in $X$ which lives over the morphisms 
		\[i\xrightarrow{\alpha} j\xrightarrow{\beta} k\] in $\CI$. Then the induced map $(gf)_{\beta_\alpha}\colon F(\beta \alpha) x_i \to x_k$ is equal to the composite 
		\[\begin{tikzcd}
			{F(\beta)F(\alpha)x_i} & {F(\beta)x_j} & {x_k}
			\arrow["{g_\beta}", from=1-2, to=1-3]
			\arrow["{F(\beta)(f_\alpha)}", from=1-1, to=1-2]
		\end{tikzcd}
		\]
	\end{lemma}
	
	\begin{proof}
		This follows immediately from the commutative diagram
		\[
		\begin{tikzcd}[column sep = large]
			{x_i} & {F(\alpha)x_i} & {F(\beta)F(\alpha)x_i} \\
			& {x_j} & {F(\beta)x_j} \\
			& {} & {x_k.} \\
			\arrow["{g_\beta}", from=2-3, to=3-3]
			\arrow["{F(\beta)(f_\alpha)}", from=1-3, to=2-3]
			\arrow["f"', from=1-1, to=2-2]
			\arrow["g"', from=2-2, to=3-3]
			\arrow[tail, from=2-2, to=2-3]
			\arrow[tail, from=1-1, to=1-2]
			\arrow["{f_\alpha}", from=1-2, to=2-2]
			\arrow[tail, from=1-2, to=1-3]
		\end{tikzcd}
		\]
		in $\Unco{F}$ which lives over the triangle $i\to j\to k$ in $\CI$ as suggested by the notation, and whose tailed morphisms are cocartesian. 
	\end{proof}
	
	\begin{remark}\label{rem:lim_of_cocart_edges}
		We would like to understand the structure maps in a limit of sections. To this end we suppose $X$ is a cocartesian fibration over $[1]$ classifying a functor $F\colon \CC\rightarrow \CD$. Then we observe that the inclusion $\CD \hookrightarrow X$ given by including the fiber over $\{1\}$ into $X$ preserves limits: given an object $C\in \CC$,
		\begin{align}\label{eq:lim_in_cocart}
			\Hom_{X}(C,\lim D_j) \simeq \Hom_{\CD}(F(C),\lim D_j) \simeq \lim \Hom_{\CD}(F(C),D_j) \simeq \lim \Hom_{X}(C,D_j). 
		\end{align}
		Now suppose that $\CC$ and $\CD$ both admit $\CJ$-shaped limits, and consider a $\CJ$-shaped diagram $\{s_j\colon [1]\rightarrow X\}_{j\in \CJ}$ of sections. By \Cref{prop:laxlim_limits}, the limit of this diagram exists in $\Fun_{[1]}([1],X)$, and is given by $\lim s_j(0)\rightarrow \lim s_j(1)$. Since $\lim s_j(1)$ is a limit in $X$, the map $\lim s_j(0)\rightarrow \lim s_j(1)$ is induced by the cone
		\[
		\lim s_j(0) \rightarrow s_j(0) \rightarrow s_j(1).
		\]	
		Factoring this through a cocartesian edge $\lim s_j(0)\rightarrow F (\lim s_j(0))$ over $0\rightarrow 1$, we obtain a cone $F (\lim s_j(0)) \rightarrow s_j(1)$, which induces a map $F(\lim s_j(0))\rightarrow \lim s_j(1)$. Furthermore the equivalence constructed in \eqref{eq:lim_in_cocart} shows that the composite of these two maps is equivalent to the map $\lim s_j(0)\rightarrow \lim s_j(1)$. Now applying \Cref{lem:pushforward_composite} we find that the cone $F (\lim s_j(0))\rightarrow s_j(1)$ is equivalent to the composite
		\[
		F(\lim s_j(0)) \rightarrow F(s_j(0)) \rightarrow s_j(1).
		\]
		In particular the map $F(\lim s_j(0))\rightarrow \lim s_j(1)$ factors as a composite 
		\[
		F(\lim s_j(0)) \rightarrow \lim F (s_j(0)) \rightarrow \lim s_j(1),
		\] where the first map is the canonical limit comparison map, and the second map is the limit of the maps $F(s_j(0))\rightarrow s_j(1)$ induced by the maps $s_j(0)\rightarrow s_j(1)$. In particular suppose each of the maps $s_j(0)\rightarrow s_j(1)$ was cocartesian. Then this second map is an equivalence, and so we conclude that the map $\lim s_j(0)\rightarrow \lim s_j(1)$ is cocartesian if and only if $F$ preserves $\CJ$-limits.
		
		We note that the dual analysis applies to colimits in cartesian fibrations over $[1]$.
	\end{remark}
	
	We can now give a sufficient condition for the inclusion of the partially lax limit into the lax limit to preserve limits.
	
	\begin{proposition}\label{prop:lim_in_part_laxlim}
		Consider a marked category $(\CI,\CW)$, and a diagram $F\colon \CI\rightarrow \Cat$. Suppose that the value of $F$ on every $i$ admits limits of shape $\CJ$ and that for every $\alpha\in \CW$ the functor $F(\alpha)$ preserves limits of shape $\CJ$. Then $\laxlimdag F$ admits limits of shape $\CJ$, and they are preserved by the inclusion $\laxlimdag F\subset \laxlim F$.
	\end{proposition}
	
	\begin{proof}
		Consider a $\CJ$-shaped diagram $\{s_j\colon \CI\rightarrow \Unco{F}\}$ in $\laxlimdag F$. We have to show that the limit of this diagram in $\laxlim F$ is again in $\laxlimdag F$. I.e. given an edge $\alpha\colon i\rightarrow i'$, we have to show that the map $\lim s_j(i)\rightarrow s_j(i')$ is cocartesian. However this can be checked by first pulling back along $\alpha\colon[1]\rightarrow \CI$, where the analysis of \Cref{rem:lim_of_cocart_edges} gives the conclusion.
	\end{proof}
	
	\begin{remark}
		Under the assumptions of the previous proposition we have shown, using the informal description of objects in a lax limit from \Cref{rem:sections_unwind}, that 
		\[
		\lim \{X_i,s_\alpha\} = \{\lim X_i, \lim s_\alpha\circ \phi\},
		\] 
		in $\laxlim F$, where $\phi\colon  F(\alpha) \lim X_i\rightarrow \lim F(\alpha) (X_i)$ is the canonical limit comparison map.
	\end{remark}
	
	Similarly we can provide sufficient conditions for the inclusion of the partially lax limit to preserve colimits. 
	
	\begin{proposition}\label{prop:colim_in_part_laxlim}
		Consider a marked category $(\CI,\CW)$, and a diagram $F\colon \CI\rightarrow \Cat^L$. Suppose that the value of $F$ on every $i
		\in \CI$ admits colimits of shape $\CJ$. Then $\laxlimdag F$ admits colimits of shape $\CJ$, and they are preserved by the inclusion $\laxlimdag F\subset \laxlim F$.
	\end{proposition}
	
	\begin{proof}
		Write $G\colon \CI^{\op}\rightarrow \Cat^R$ for the diagram of right adjoint associated to $F$. Applying the dual analysis of \Cref{rem:lim_of_cocart_edges}, we find that given a $\CJ$-shaped diagram $\{s_j\colon \CI\rightarrow \Unco{F}\}$ in $\laxlimdag F$ and a map $\alpha\colon i\rightarrow i'$ in $\CW$, the induced map $\colim s_j(i)\rightarrow \colim s_j(i')$ factors as the map
		\[ 
		\phi\colon \colim s_j(i)\rightarrow \colim G(\alpha)(s_j(i')) \rightarrow G(\alpha)(s_j(i'))
		\] followed by a cartesian edge in $\Unco{F}$ over $\alpha$. The map $F(\alpha)\colim s_j(i)\rightarrow \colim s_j(i')$ given by instead factoring $\colim s_j(i)\rightarrow \colim s_j(i')$ through a cocartesian edge is adjoint to $\phi$. In particular we compute that it is given by the composite
		\[
		F(\alpha)(\colim s_j(i))\xrightarrow{\sim} \colim F(\alpha) (s_j(i)) \rightarrow \colim s_j(i').
		\]
		Because the original maps $s_j(i)\rightarrow s_j(i')$ were cocartesian, this is an equivalence. We conclude that $\colim s_j(i)\rightarrow \colim s_j(i')$ is again cocartesian.
	\end{proof}
	
	\begin{remark}
		Under the assumptions of the previous proposition we have shown, using the informal description of objects in a lax limit, that 
		\[
		\colim \{X_i,s_\alpha\} = \{\colim X_i, \colim s_\alpha \circ \phi^{-1}\},
		\] 
		where $\phi\colon \colim F(\alpha) X_i\rightarrow F(\alpha) \colim X_i$ is the canonical colimit comparison map. 
	\end{remark}

	\subsection{Adjunctions of partially lax limits}\label{subsec:adj_of_part}
	Let $(\CI,\CW)$ be a relative category, and consider two functors $F,G\colon \CI\rightarrow \Cat$. Suppose that one has a commutative diagram 
	\[
	\begin{tikzcd}
		\Unco{F} \arrow[rr,"H"] \arrow[rd,"p"']&		&	\Unco{G} \arrow[ld, "q"] \\
		&	\CI	&
	\end{tikzcd}
	\]
	such that $H$ preserves cocartesian edges which lie over an edge in $\CW$. We call such a commutative diagram a \emph{partially lax transformation} from $F$ to $G$.
	
	\begin{remark}
		Note that if $H$ in fact preserves all cocartesian edges then it corresponds via straightening to an honest natural transformation. By weakening this condition we obtain laxly commuting naturality squares, explaining the terminology. For further justification see \cite{HHLN1}.
	\end{remark}
	
	Observe that $H$ induces a functor
	\[
	\laxlimdag_{(\CI,\CW)}F \coloneqq \Fun^{\CW\text{-co}}_\CI(\CI,\Unco{F}) \xrightarrow{H_*} \Fun^{\CW\text{-co}}_\CJ(\CI,\Unco{G}) \eqqcolon \laxlimdag_{(\CI,\CW)} G.
	\] 
	To summarize, partially lax limits are functorial in partially lax natural transformations. Furthermore we note that partially lax limits are also clearly functorial in natural transformations $H\Rightarrow H'$ of partially lax natural transformations $\Unco{F}\rightarrow \Unco{G}$ which lie over the identity of $\CI$. We will now give a sufficient condition for the functor $H^*$ constructed above to admit a right adjoint, for which we first recall the following result.
	
	\begin{lemma}\label{lem:adj_on_unstr}
		Suppose $F,G\colon \CI\rightarrow \Cat$ are two diagram, and consider a natural transformation $\eta\colon F\Rightarrow G$ which is pointwise a left adjoint. Then the functor $H\colon \Unco{F}\rightarrow \Unco{G}$ encoding $\eta$ is a left adjoint. The associated right adjoint $J\colon \Unco{G}\rightarrow \Unco{F}$ is again a functor over $\CI$ and is given on the fiber over $i\in \CI$ by the right adjoint of $\eta_i$. Furthermore the unit and counit of the adjunction $H\dashv J$ live over the identity natural transformation on $\id_\CI$. Finally $J$ preserves cocartesian edges over $f\colon i\rightarrow j$ if and only if the commutative square 
		\[
		\begin{tikzcd}
			{F(i)} & {F(j)} \\
			{G(i)} & {G(j)}
			\arrow["{\eta_i}"', from=1-1, to=2-1]
			\arrow["{F(f)}", from=1-1, to=1-2]
			\arrow["{G(f)}", from=2-1, to=2-2]
			\arrow["{\eta_j}", from=1-2, to=2-2]
		\end{tikzcd}
		\]
		is left adjointable.
	\end{lemma}
	
	\begin{proof}
		This is the dual of \cite{HA}*{Proposition 7.3.2.6}.
	\end{proof}
	
	\begin{remark}
		The condition that the right adjoint $J$ of $H$ again lies over $\CI$ and that the unit and counit natural transformations of the adjunction $H\dashv J$ live over the identity natural transformation of $\CI$ can be summarized by saying that $H$ is a left adjoint in the 2-category $\Cat_{/\CI}$. This is called a \emph{relative left adjoint} in \cite{HA}.
	\end{remark}
	
	\begin{remark}
		the unstraightening of a parametrized left adjoint $L\colon \CC\rightarrow \CD$ of $T$-categories is a  relative left adjoint over $T$.
	\end{remark}
	
	From the previous result we immediately obtain the following proposition.
	
	\begin{proposition}\label{prop:adj_part_lax_lim}
		Consider a marked category $(\CI,\CW)$, two diagrams $F,G\colon \CI\rightarrow \Cat$, and a natural transformation $L\colon F\Rightarrow G$ such that each $L_i\colon F(i)\rightarrow G(i)$ is a left adjoint and the square
		\[
		\begin{tikzcd}
			{F(i)} & {G(i)} \\
			{F(j)} & {G(j)}
			\arrow["{L_i}", from=1-1, to=1-2]
			\arrow["{F(f)}"', from=1-1, to=2-1]
			\arrow["{G(f)}", from=1-2, to=2-2]
			\arrow["{L_j}"', from=2-1, to=2-2]
		\end{tikzcd}
		\]
		is left adjointable for $f\in \CW$. Then the functor $L_*\colon \laxlimdag F \rightarrow \laxlimdag G$ is a left adjoint, with right adjoint given by postcomposing by the partially lax natural transformation given by passing to right adjoints pointwise, as in \Cref{lem:adj_on_unstr}.
	\end{proposition}
	
	\begin{proof}
		By \Cref{lem:adj_on_unstr} the unstraightening of $L$ is a left adjoint in $\Cat_{/\CI}$, whose right adjoint preserves cocartesian edges over $\CW$. By applying the functoriality of partially lax limits in partially lax natural transformations we obtain the required adjunction.
	\end{proof}
	
	Next we introduce the contravariant functoriality of partially lax limits. Consider a functor of marked categories $h\colon (\CI,\CW)\rightarrow (\CJ,\CW')$ and a functor $F\colon \CJ\rightarrow \Cat$. Given a section $s\colon \CJ \rightarrow \Unco{F}$ we may precompose with the functor $h$ to obtain a functor $\CI\rightarrow \Unco{F}$, which we may interpret as a section $t\colon \CI \rightarrow \Unco{F}\times_\CJ \CI$ of the pullback $\Unco{F}\times_\CJ \CI\rightarrow \CI$. Recall that $\Unco{F}\times_\CJ \CI\rightarrow \CI$ is a cocartesian fibration which classifies the functor $F\circ h$. Furthermore an edge in $\Unco{F}\times_\CJ \CI$ is cocartesian if it is in $\Unco{F}$, and so we conclude  from the fact that $h$ is a functor of marked categories that if $s$ sent edges in $\CW'$ to cocartesian edges of $\Unco{F}$ then $t$ sends edges of $\CW$ to cocartesian edges of $\Unco{F\circ h}$. In total we obtain a functor
	\[
	\Fun^{\dagger}_\CJ(\CJ,\Unco{F}) \xrightarrow{h^*} \Fun^{\dagger}_\CJ(\CI,\Unco{F})\simeq \Fun^{\dagger}_{\CI}(\CI,\Unco{Fh}).
	\]
	Summarizing, partially lax limits are contravariantly functorial in functors of marked categories. We will again give a sufficient condition for the functor $h^*\colon \laxlimdag F\rightarrow \laxlimdag Fh$ constructed above to have a left and right adjoint. To do this we begin with a general categorical result.
	
	\begin{proposition}\label{prop:adj_cart}
		Consider a diagram
		\[\begin{tikzcd}
			X & Y \\
			\CI & \CJ
			\arrow["p"', from=1-1, to=2-1]
			\arrow["q", from=1-2, to=2-2]
			\arrow[""{name=0, anchor=center, inner sep=0}, "L", shift left=2, from=2-1, to=2-2]
			\arrow[""{name=1, anchor=center, inner sep=0}, "R", shift left=2, from=2-2, to=2-1]
			\arrow[""{name=2, anchor=center, inner sep=0}, "{\bbL}", shift left=2, from=1-1, to=1-2]
			\arrow[""{name=3, anchor=center, inner sep=0}, "{\bbR}", shift left=2, from=1-2, to=1-1]
			\arrow["\dashv"{anchor=center, rotate=-90}, draw=none, from=0, to=1]
			\arrow["\dashv"{anchor=center, rotate=-90}, draw=none, from=2, to=3]
		\end{tikzcd}\]
		of categories such that $p$ and $q$ are cartesian fibrations, both possible squares commute, and $p$ and $q$ map the unit and counit of $\bbL\dashv \bbR$ to that of $L\dashv R$. Suppose $X$ classifies the functor $G\colon \CI^{\op}\rightarrow \Cat$. Then given an object $i\in \CI$, the functor
		$\bbL\colon X_i\rightarrow Y_{L(i)}$ admits a right adjoint given by the composite
		\[Y_{L(i)}\xrightarrow{\bbR} X_{RL(i)} \xrightarrow{G(\eta)} X_{i}.\]
	\end{proposition}
	
	\begin{proof}
		Suppose $x,y$ are objects of $X_i$ and $Y_{L(i)}$ respectively. By assumption the bottom square of the following diagram commutes:
		\[\begin{tikzcd}
			{\Hom_{Y_{L(i)}}(\bbL(x),y)} & {\Hom_{X_i}(x,G(\eta)\bbR(y))} \\
			{\Hom_Y(\bbL(x), y)} & {\Hom_X(x, \bbR(y))} \\
			{\Hom_\CJ(L(i),L(i))} & {\Hom_\CI(i, RL(i))}
			\arrow["q", from=2-2, to=3-2]
			\arrow["p", from=2-1, to=3-1]
			\arrow["\sim", from=3-1, to=3-2]
			\arrow["\sim", from=2-1, to=2-2]
			\arrow[from=1-1, to=2-1]
			\arrow[dashed, from=1-1, to=1-2]
			\arrow[from=1-2, to=2-2]
		\end{tikzcd}\]
		By \cite{HTT}*{Proposition 2.4.4.2}, the fiber over $\id_{L(i)}$ and $\eta$ of $p$ and $q$ respectively are given by the top two spaces of the diagram, and therefore we obtain the dashed equivalence.
	\end{proof}
	
	\begin{example}\label{ex:Ar_cat_adj}
		Consider an adjunction $L\colon \CC\rightleftarrows \CD\cocolon R$ between two categories admitting pullbacks. Applying the previous proposition to the square
		\[\begin{tikzcd}
			{\Ar(\CC)} & {\Ar(\CD)} \\
			\CC & \CD
			\arrow["{\ev_1}", from=1-1, to=2-1]
			\arrow["{\ev_1}", from=1-2, to=2-2]
			\arrow[""{name=0, anchor=center, inner sep=0}, "L", shift left=2, from=2-1, to=2-2]
			\arrow[""{name=1, anchor=center, inner sep=0}, "R", shift left=2, from=2-2, to=2-1]
			\arrow[""{name=2, anchor=center, inner sep=0}, "\Ar(L)", shift left=2, from=1-1, to=1-2]
			\arrow[""{name=3, anchor=center, inner sep=0}, "\Ar(R)", shift left=2, from=1-2, to=1-1]
			\arrow["\dashv"{anchor=center, rotate=-90}, draw=none, from=0, to=1]
			\arrow["\dashv"{anchor=center, rotate=-90}, draw=none, from=2, to=3]
		\end{tikzcd}\]
		we conclude that $L\colon \CC_{/x}\to \CD_{/L(x)}$ admits a right adjoint, given by the composite
		\[\CD_{/L(x)} \xrightarrow{R} \CC_{/RL(x)} \xrightarrow{\eta^*} \CC_{/x}.\] 
		This is a well-known fact, proven as \cite{HTT}*{Proposition 5.2.5.1} for example.
	\end{example}
	
	We also record the dual proposition:
	
	\begin{proposition}\label{prop:adj_cocart}
		Consider a diagram
		\[\begin{tikzcd}
			X & Y \\
			\CI & \CJ
			\arrow["p"', from=1-1, to=2-1]
			\arrow["q", from=1-2, to=2-2]
			\arrow[""{name=0, anchor=center, inner sep=0}, "L", shift left=2, from=2-1, to=2-2]
			\arrow[""{name=1, anchor=center, inner sep=0}, "R", shift left=2, from=2-2, to=2-1]
			\arrow[""{name=2, anchor=center, inner sep=0}, "{\bbL}", shift left=2, from=1-1, to=1-2]
			\arrow[""{name=3, anchor=center, inner sep=0}, "{\bbR}", shift left=2, from=1-2, to=1-1]
			\arrow["\dashv"{anchor=center, rotate=-90}, draw=none, from=0, to=1]
			\arrow["\dashv"{anchor=center, rotate=-90}, draw=none, from=2, to=3]
		\end{tikzcd}\]
		of categories such that $p$ and $q$ are cocartesian fibrations and both possible squares commute. Suppose $q$ classifies the functor $F\colon \CJ \rightarrow \Cat$. Then given an object $j\in \CJ$, the functor
		$\bbR\colon Y_j\rightarrow X_{R(j)}$ admits a left adjoint given by the composite
		\[X_{R(j)}\xrightarrow{\bbL} Y_{LR(j)} \xrightarrow{F(\epsilon)} Y_{j}.\qednow\]
	\end{proposition}
	
	\begin{example}
		Applying the previous proposition to the situation of \Cref{ex:Ar_cat_adj}
		we conclude that $R\colon \CD_{/x}\to \CC_{/R(x)}$ admits a left adjoint, given by the composite
		\[\CC_{/R(x)} \xrightarrow{L} \CC_{/LR(x)} \xrightarrow{\epsilon_!} \CC_{/x}.\] 
	\end{example}
	
	We now apply these results to construct adjoints to the contravariant functoriality of partially lax limits. 
	
	\begin{proposition}\label{prop:right_adj_to_res_in_lim}
		Consider a functor $F\colon \CI\rightarrow \Cat^{L}$ and an adjunction $L\colon \CI \rightleftarrows \CJ\cocolon R$. Then $R^*\colon \laxlim F\rightarrow \laxlim FR$ is a left adjoint.
	\end{proposition}
	
	\begin{proof}
		We can build the square
		\[\begin{tikzcd}
			{\Fun(\CI,\Unco{F})} & {\Fun(\CJ,\Unco{F})} \\
			{\Fun(\CI,\CI)} & {\Fun(\CJ,\CI).}
			\arrow["{p_*}"', from=1-1, to=2-1]
			\arrow["{p_*}", from=1-2, to=2-2]
			\arrow[""{name=0, anchor=center, inner sep=0}, "R^*", shift left=2, from=2-1, to=2-2]
			\arrow[""{name=1, anchor=center, inner sep=0}, "L^*", shift left=2, from=2-2, to=2-1]
			\arrow[""{name=2, anchor=center, inner sep=0}, "{L^*}", shift left=2, from=1-2, to=1-1]
			\arrow[""{name=3, anchor=center, inner sep=0}, "{R^*}", shift left=2, from=1-1, to=1-2]
			\arrow["\dashv"{anchor=center, rotate=-90}, draw=none, from=3, to=2]
			\arrow["\dashv"{anchor=center, rotate=-90}, draw=none, from=0, to=1]
		\end{tikzcd}\]
		By \cite{HTT}*{Corollary 5.2.2.5} the cocartesian fibration $p\colon \Unco{F}\rightarrow \CI$ classifying $F$ is also a cartesian fibration. By \cite{HTT}*{Proposition 3.1.2.1(1)} the functors $p_*$ are both also cartesian fibrations, and therefore this square is of the form required to apply \Cref{prop:adj_cart}. In particular considering the object $\id_\CI$ in $\Fun(\CI,\CI)$ we obtain an adjunction
		\[R^*\colon \Fun_\CI(\CI,\Unco{F}) \rightleftarrows \Fun_{\CI}(\CJ,\Unco{F})\cocolon \Theta(\eta^*) L^*,\] 
		where $\Theta$ refers to the functor classified by $p_*$. Note that the left hand category is equal to sections of $\Unco{F}$ while the right-hand side is equivalent to sections of $\Unco{FR}$. These are equivalent to $\laxlim F$ and $\laxlim FR$ respectively, and so we conclude.
	\end{proof}
	
	\begin{remark}
		We continue to use the same notation as in the proposition above, and further write $G\colon \CI^{\op}\rightarrow \Cat^{R}$ for the diagram of right adjoints associated to $F$. It is potentially illuminating to informally summarize the adjunction constructed above using the notation of \Cref{rem:sections_unwind}. In this notation, the functor $R^*$ sends the object
		\[
		\{i\mapsto X_i\in F(i)\quad \alpha \mapsto f_{\alpha}\colon F(\alpha) X_i\rightarrow X_i'\}
		\] 
		in $\laxlim F$ to the object 
		\[
		\{j\mapsto X_{R(j)} \in F(R(i)) \quad \beta \mapsto f_{R(\beta)}\colon F(R(\beta)) X_{R(i)}\rightarrow X_{R(j)}\}
		\] 
		in $\laxlim FR$. We will see in the proof of the next proposition that the right adjoint sends the object
		\[
		\{j\mapsto Y_j \in FR(j),\quad \beta\mapsto f_{\beta}\colon FR(\beta) Y_j\rightarrow Y_{j'}\}
		\] 
		to the object
		\[
		\hspace{-14.5001pt}\{i\mapsto G(\eta_i) Y_{L(i)} \in F(i), \quad \alpha \mapsto \big[F(\alpha)G(\eta_i) Y_{L(i)} \xrightarrow{BC} G(\eta_{i'})F(RL(\alpha))Y_{L(i)} \xrightarrow{G(\eta_{i'}) f_{L(\alpha)}} G(\eta_{i'})Y_{L(i')}\big]\},
		\]
		where $\mathrm{BC}$ denotes the Beck--Chevalley transformation.
	\end{remark}
	
	This informal description suggests when the adjunction constructed above restricts to one between \emph{partially} lax limits. We make this precise in the following proposition.
	
	\begin{proposition}\label{prop:right_adj_to_res_in_part}
		In the situation of \Cref{prop:right_adj_to_res_in_lim}, suppose further that $\CI$ and $\CJ$ are marked by $\CW$ and $\CW'$ respectively, that $L$ and $R$ both preserve marked edges, and that the square
		\[\begin{tikzcd}
			{\CC_i} & {\CC_{RL(i)}} \\
			{\CC_{i'}} & {\CC_{RL(i')}}
			\arrow["{F(RL(\alpha))}", from=1-2, to=2-2]
			\arrow["{F(\eta_i)}", from=1-1, to=1-2]
			\arrow["{F(\eta_{i'})}"', from=2-1, to=2-2]
			\arrow["{F(\alpha)}"', from=1-1, to=2-1]
		\end{tikzcd}\]
		induced by the naturality square of $\eta$ is left adjointable whenever $\alpha$ is marked, Then $R^*$ and its adjoint restrict to an adjunction between partially lax limits.
	\end{proposition}
	
	\begin{proof}
		Recall that we have constructed an adjunction 
		\[
		R^*\colon \laxlim F \rightleftarrows \laxlim FR\cocolon \Theta(\eta^*)L^*.
		\] 
		We want to show that both functors restrict to partially lax limits. In the case of $R^*$ this is clear, because $R$ is a functor of marked categories. However showing that $\Theta(\eta^*)L^*$ restricts appropriately is more subtle. We fix an object $s\in \laxlimdag FR$, i.e.~a functor $s\colon \CJ\rightarrow \Unco{F}$ living over $R$ which sends edges in $\CW'$ to cocartesian edges. Recall that by \cite{HTT}*{Proposition 3.1.2.1(2)} a morphism in $\Fun(\CI,\Unco{F})$ is $p_*$-cartesian if and only if each component is $p$-cartesian, and therefore the value of $\Theta(\eta^*)s(L(i))$ at $i$ is equal to the source of the essentially unique cartesian arrow $G(\eta_i)s(L(i)) \twoheadrightarrow s(L(i))$ which lives over $\eta\colon i \rightarrow LR(i)$, i.e.~the image of $Y(L(i))$ under the functor $G(\eta_i)$. Similarly given a map $\alpha\colon i\rightarrow i'$ in $\CI$, the map $\Theta(\eta^*)L^*s(\alpha)$ is homotopic to the unique dotted map
		\[\begin{tikzcd}
			{G(\eta_i)s(L(i))} & {s(L(i))} \\
			{G(\eta_j)s(L(i'))} & {s(L(i'))}
			\arrow[from=1-2, to=2-2]
			\arrow[two heads, from=1-1, to=1-2]
			\arrow[two heads, from=2-1, to=2-2]
			\arrow["\phi", dashed, from=1-1, to=2-1]
		\end{tikzcd}\] for which the resulting square commutes and lives over the square
		\[\begin{tikzcd}
			i & {RL(i)} \\
			{i'} & {RL(i').}
			\arrow["{\eta_{i'}}", from=2-1, to=2-2]
			\arrow["{\eta_i}", from=1-1, to=1-2]
			\arrow["{RL(f)}", from=1-2, to=2-2]
			\arrow["f"', from=1-1, to=2-1]
		\end{tikzcd}\]
		in $\CI$. However we can build the following commutative diagram 
		\[\begin{tikzcd}
			{G(\eta_i)s(L(i))} & {s(L(i))} \\
			{F(\alpha)G(\eta_i)s(L(i))} \\
			{G(\eta_{i'})F(RL(\alpha))s(L(i))} & {F(\alpha)s(L(i))} \\
			{G(\eta_{i'})s(L(i'))} & {s(L(i'))}
			\arrow[two heads, from=1-1, to=1-2]
			\arrow[two heads, from=4-1, to=4-2]
			\arrow[tail, from=1-2, to=3-2]
			\arrow["{f_{L(\alpha)}}", from=3-2, to=4-2]
			\arrow[tail, from=1-1, to=2-1]
			\arrow["{G(\eta_{i'})f_{L(\alpha)}}"', from=3-1, to=4-1]
			\arrow["BC"', from=2-1, to=3-1]
			\arrow[two heads, from=3-1, to=3-2]
		\end{tikzcd}\]
		in $X$ which lives over 
		\[\begin{tikzcd}
			i & {RL(i)} \\
			{i'} \\
			{i'} & {RL(i')} \\
			{i'} & {RL(i').}
			\arrow[from=1-1, to=1-2]
			\arrow[from=4-1, to=4-2]
			\arrow["{RL(\alpha)}", from=1-2, to=3-2]
			\arrow[Rightarrow, no head, from=3-2, to=4-2]
			\arrow["\alpha"', from=1-1, to=2-1]
			\arrow[Rightarrow, no head, from=3-1, to=4-1]
			\arrow[from=3-1, to=3-2]
			\arrow[Rightarrow, no head, from=2-1, to=3-1]
		\end{tikzcd}\] In this diagram the two-headed arrows are cartesian and the tailed arrows are cocartesian. By conclude that the composite along the left hand side is homotopic to $\phi$. The identification of the map ${F(\alpha)G(\eta_i)s(L(i))} \rightarrow {G(\eta_{i'})F(RL(\alpha))s(L(i))}$ with the Beck--Chevalley transformation of the square 
		\[\begin{tikzcd}
			{\CC_{RL(i)}} & {\CC_i} \\
			{\CC_{RL(i')}} & {\CC_{i'}}
			\arrow["{F(RL(\alpha))}", from=1-2, to=2-2]
			\arrow["{F(\eta_i)}", from=1-1, to=1-2]
			\arrow["{F(\eta_{i'})}"', from=2-1, to=2-2]
			\arrow["{F(\alpha)}"', from=1-1, to=2-1]
		\end{tikzcd}\]
		is given in \cite{HHLN1}*{Proposition 3.2.7}. Our assumptions on $L$ and $R$ imply that this map, as well as the map $G(\eta_{i'})f_{L(\alpha)}$, are equivalences. Therefore the map $\phi$ is cocartesian, and we conclude that $\Theta(\eta^*) L^*$ preserves partially cocartesian sections. 
	\end{proof}
	
	We can also give a sufficient condition for the contravariant restriction along a functor of marked categories to admit a left adjoint.
	
	\begin{proposition}\label{prop:left_adj_to_res_in_lim}
		Consider a functor $F\colon \CJ\rightarrow \Cat$ and an adjunction $L\colon \CI \rightleftarrows \CJ\cocolon R$. Then $L^*\colon \laxlim F\rightarrow \laxlim FL$ is a right adjoint. If we furthermore suppose that $\CI$ and $\CJ$ are marked and the functors $L$ and $R$ preserve the markings, then $L^*$ and its right adjoint preserve objects of the partially lax limit. In particular they restrict to an adjunction between partially lax limits.
	\end{proposition}
	
	\begin{proof}
		We can build the square
		\[\begin{tikzcd}
			{\Fun(\CI,\Unco{F})} & {\Fun(\CJ,\Unco{F})} \\
			{\Fun(\CI,\CJ)} & {\Fun(\CJ,\CJ).}
			\arrow["{p_*}"', from=1-1, to=2-1]
			\arrow["{p_*}", from=1-2, to=2-2]
			\arrow[""{name=0, anchor=center, inner sep=0}, "R^*", shift left=2, from=2-1, to=2-2]
			\arrow[""{name=1, anchor=center, inner sep=0}, "L^*", shift left=2, from=2-2, to=2-1]
			\arrow[""{name=2, anchor=center, inner sep=0}, "{L^*}", shift left=2, from=1-2, to=1-1]
			\arrow[""{name=3, anchor=center, inner sep=0}, "{R^*}", shift left=2, from=1-1, to=1-2]
			\arrow["\dashv"{anchor=center, rotate=-90}, draw=none, from=3, to=2]
			\arrow["\dashv"{anchor=center, rotate=-90}, draw=none, from=0, to=1]
		\end{tikzcd}
		\]
		Note this square is of the form required by \Cref{prop:adj_cocart}. In particular considering the object $\id_\CJ$ in $\Fun(\CJ,\CJ)$ we obtain an adjunction
		\[\Gamma(\epsilon^*) R^*\colon \Fun_{L}(\CI,\Unco{F}) \rightleftarrows \Fun_{\CJ}(\CJ,\Unco{F})\cocolon  L^*,\] 
		where $\Gamma$ is the functor associated to the cocartesian fibration $p_*$. Note that the right-hand side is equal to sections of $\Unco{F}$ while the left-hand side is equivalent to sections of $\Unco{FL}$. These are equivalent to $\laxlim F$ and $\laxlim FL$ respectively, and so we conclude the first statement. The second statement follows from an analysis of the functor $\Gamma(\epsilon^*) R^*$, as in \Cref{prop:right_adj_to_res_in_part}.
	\end{proof}
	
	\begin{remark}
		Similarly to before one can show that the left adjoint to $L^*$ sends an object of the form
		\[
		\{i\mapsto Y_i \in FL(i),\quad \alpha\mapsto f_{\alpha}\colon FL(\alpha) Y_i\rightarrow Y_{i'}\}
		\] 
		in $\laxlim FL$ to an object of the form
		\[
		\hspace{-11.912pt}\{j\mapsto F(\epsilon_j) Y_{R(j)} \in F(j), \quad \beta \mapsto \big[F(\beta)F(\epsilon_i) Y_{R(j)} \simeq F(\epsilon_{j'})F(LR(\beta))Y_{R(j)} \xrightarrow{F(\epsilon_{j'}) f_{L(\beta)}} F(\epsilon_{j'})Y_{L(i')}\big]\}.
		\] in $\laxlim F$.
	\end{remark}
	
	\section{Freely adding parametrized colimits and globalization}\label{sec:Glob}
	
	We fix an orbital pair $(T,S)$ for the remainder of the section. Note that there is an obvious forgetful functor \[\fgt\colon \PrL{T}{\L}\rightarrow \PrL{T}{S}\] which exhibits $\PrL{T}{\L}$ as a non-full subcategory of $\PrL{T}{S}$. In this section we will construct the \emph{S-relative cocompletion} $\Free(\CC)$ of an $S$-presentable $T$-category $\CC$, and exhibit $\Free$ as a left adjoint to $\fgt\colon \PrL{T}{\L}\rightarrow \PrL{T}{S}$.
	
	\begin{definition}
		We define a functor 
		\[\Free\colon \Cat_T \rightarrow \Fun(\bbF_T^{\op},\Cat)\] 
		via the assigment $\Free(\CC)_X = \laxlimdag (\pi_X^* \CC),$ where $\pi_X^* \CC$ denotes the functor 
		\[\pi_X^*\CC \colon ({\bbF_{T}}_{/X})^{\op}\xrightarrow{\pi_X} \bbF_T^{\op}\xrightarrow{\CC} \Cat\]
		and an edge in $({\bbF_{T}}_{/X})^{\op}$ is marked if and only if its projection to $\bbF_T^{\op}$ lands in $\bbF_S^{\op}$. The functoriality of this assignment in $\bbF_T^{\op}$ is induced by the contravariant functoriality of partially lax limits applied to the postcomposition functoriality of the slices ${\bbF_{T}}_{/X}$.
	\end{definition}
	
	\begin{remark}\label{rem:other_description}
		Note that there is a functor \[({\bbF_T}_{/-})^{\op}\colon \bbF_T \rightarrow \Cat^\dagger_{/\bbF_T^{\op}}, \quad X\mapsto ({\bbF_T}_{/X})^{\op},\] where $({\bbF_T}_{/X})^{\op}$ is marked by the subcategory of edges whose projection to $\bbF_T^{\op}$ lies in $\bbF_S^{\op}$. Then one can equivalently define $\Free(\CC)$ as the composite of $({\bbF_T}_{/-})^{\op}$ and the contravariant functor \[\Fun_{\bbF_T^{\op}}^{\dagger}(-,\Unco{\CC})\colon \Cat^\dagger_{/\bbF_T^{\op}}\rightarrow \Cat,\] where $\Unco{\CC}$ is marked as usual by the cocartesian edges living over edges of $\bbF_S^{\op}$.
	\end{remark}
	
	\begin{remark}
		Consider $X\in \bbF_T$ and note that ${\bbF_T}_{/X}$ is the finite coproduct completion of $T_{/X}$, the category of elements of $X\in \PSh(T)$. Therefore by \Cref{prop:lim_of_lim_ext} we obtain that
		\[
		\Free(\CC)_X \simeq \laxlimdag((T_{/X})^{\op}\rightarrow T^{\op}\xrightarrow{\CC} \Cat).
		\]
	\end{remark}
	
	We will make use of all three descriptions of $\Free(\CC)_{X}$ throughout this section. Let us begin by showing that $\Free$ restricts to a functor from $S$-presentable $T$-categories to $T$-presentable $T$-categories. 
	
	\begin{lemma}
		The functor $\Free$ factors through $\Cat_T\subset \Fun(\bbF_T^{\op},\Cat)$.
	\end{lemma}
	
	\begin{proof}
		Note that ${T}_{/\coprod X_i} = \coprod {T}_{/X_i}$. Therefore \Cref{rem:lim_on_colim_diagram}, together with the previous remark, implies the desired result.
	\end{proof}	
	
	\begin{remark}
		Recall that $\Cat_T$ is canonically a 2-category via the equivalence $\Cat_T\simeq \Fun(T^{\op},\Cat)$. We observe that $\Free$ is a functor of 2-categories. This is easily seen from the description of \Cref{rem:other_description}.
	\end{remark}
	
	\begin{theorem}\label{thm:Free_is_presentable}
		$\Free$ restricts to a functor 
		\[\Free \colon \PrL{T}{S}\rightarrow \PrL{T}{\L}.\]
	\end{theorem}
	
	\begin{proof}
		Let $\CC$ be an $S$-presentable $T$-category. We have to show that $\Free(\CC)$ is a $T$-presentable $T$-category. As a first step we show that $\Free(\CC)\colon \bbF_T^{\op}\rightarrow \Cat$ factors through $\PrL{}{\L}$. We first note that by an evident generalization of \cite{HTT}*{Proposition 5.5.3.17} each category $\Free(\CC)_X$ is presentable. By \Cref{prop:colim_in_part_laxlim} we conclude that colimits in $\laxlimdag \pi_X^* \CC$ are computed fiberwise for all $X\in \bbF_T$. In particular the restriction functors
		\[
		\laxlimdag \pi_X^*\CC\rightarrow \laxlimdag \pi_Y^*\CC
		\] clearly preserve colimits. By the adjoint functor theorem, proven as \cite{HTT}*{Corollary 5.5.2.9}, $f^*\colon \Free(\CC)_Y \rightarrow \Free(\CC)_X$ admits a right adjoint, and so $\Free(\CC)$ factors through $\PrL{}{\L}$.
		
		Next we show that the functors $f^*\colon \Free(\CC)_Y\rightarrow \Free(\CC)_X$ admit left adjoints for all morphisms $f\colon X\rightarrow Y$ in $\mathbb{F}_T$, and that the squares in \Cref{def:S-presentable}(3) are left adjointable. However by \Cref{prop:lim_in_part_laxlim} the functors $f^*$ also preserves limits for every $f\colon X\rightarrow Y$ in $\bbF_T$, and so another application of the adjoint functor theorem implies that it admits a left adjoint. Therefore all that remains is to prove that the required squares are left adjointable. To do this we will explicitly describe the right adjoint of $f^*$. First we observe that because $T$ is orbital, the functor
		\[(f_!)^{\op}\colon ({\bbF_{T}}_{/X})^{\op} \rightarrow ({\bbF_{T}}_{/Y})^{\op}\] has a left adjoint $(f^*)^{\op}$ given by pulling back, and that both $(f_!)^{\op}$ and $(f^*)^{\op}$ preserve marked edges. Furthermore by the pasting law for pullbacks the square
		\[
		\begin{tikzcd}
			{Z\times_{Y} X } & {Z} \\
			{Z' \times_{Y} X} & {Z'}
			\arrow["\pi_1", from=1-1, to=1-2]
			\arrow["g \times_{Y} X"', from=1-1, to=2-1]
			\arrow["\pi_1", from=2-1, to=2-2]
			\arrow["g", from=1-2, to=2-2]
		\end{tikzcd}
		\]
		is a pullback square in $\bbF_T$ for all $g\colon Z\rightarrow Z'$ in $\bbF_T$. This implies that the square
		\[
		\begin{tikzcd}
			{\CC_Z} & {\CC_{Z\times_{Y} X} } \\
			{\CC_{Z'}} & {\CC_{Z' \times_{Y} X}} 
			\arrow["(\pi_1)^*", from=1-1, to=1-2]
			\arrow["(g \times_{Y} X)^*", from=1-2, to=2-2]
			\arrow["(\pi_1)^*", from=2-1, to=2-2]
			\arrow["g^*"', from=1-1, to=2-1]
		\end{tikzcd}
		\]
		is left adjointable whenever $g$ is in $\bbF_S$ because $\CC$ is $S$-presentable. Therefore \Cref{prop:right_adj_to_res_in_part} gives an explicit description of the right adjoint $f_*$ of the restriction functor
		\[f^*\colon \laxlimdag (\pi_Y^* \CC ) \rightarrow \laxlimdag (\pi_X^* \CC).\] 
		Informally, $f_*$ sends the object
		\[\{h\colon Z\rightarrow X \mapsto C_{h} \in \CC_{Z},\quad [g\colon h \to h'] \mapsto \lambda_{g}\colon g^* C_{h'}\rightarrow C_{h}\}\] 
		to the object 
		\begin{align*}
			\hspace{-15.89pt}\{h\colon Z\rightarrow Y \mapsto &(\pi_1)_* C_{h\times_Y X} \in \CC_Z, \\
			&[g\colon h \rightarrow h'] \mapsto \big[ g^*(\pi_1)_* C_{h'\times_Y X} \xrightarrow{BC} (\pi_{1})_* (g\times_X Y)^* C_{h'\times_Y X} \xrightarrow{(\pi_1)_* \lambda_{g\times_X Y}} (\pi_1)_*Y_{h\times_X Y}\big]\}.
		\end{align*}
		We can now show that the required squares are left adjointable: Given a pullback square
		\[\begin{tikzcd}
			X' & X \\
			{Y'} & Y
			\arrow["{f'}"', from=1-1, to=2-1]
			\arrow["g", from=2-1, to=2-2]
			\arrow["f"', from=1-2, to=2-2]
			\arrow["{g'}", from=1-1, to=1-2]
		\end{tikzcd}\]
		in $\bbF_T$ it suffices by passing to total mates to prove that the Beck-Chevalley transformation filling the square
		\[\begin{tikzcd}
			{\Free(\CC)_{X'}} & {\Free(\CC)_{X}} \\
			{\Free(\CC)_{Y'}} & {\Free(\CC)_{Y}}
			\arrow["{(f')^*}", from=2-1, to=1-1]
			\arrow["{f^*}"', from=2-2, to=1-2]
			\arrow["{(g')_*}", from=1-1, to=1-2]
			\arrow["{g_*}", from=2-1, to=2-2]
			\arrow[shorten <=9pt, shorten >=9pt, Rightarrow, from=2-2, to=1-1]
		\end{tikzcd}\]
		is an equivalence. 
		However unwinding the definition of the Beck--Chevalley transformation we find that on a section $s\colon ({\bbF_T}_{/Y})^{\op}\rightarrow \Unco{\pi^*_Y\, \CC}$ it is given at $h\colon Z\rightarrow X$ by applying $(\pi_1)_*$ to the map $\beta^* (s(fh\times_{Y} Y'))\xrightarrow{\sim} s(f'\circ (h\times_X X'))$ induced by the base-change equivalence $\beta\colon fh\times_Y Y'\xrightarrow{\sim} f'\circ(h\times_X X')$, i.e.~ the morphism witnessing the equivalence of $fh\times_Y Y'$ and $f'\circ (h\times_X X')$ in the slice ${\bbF_T}_{/Y}$. In particular the Beck-Chevalley transformation is an equivalence. Altogether we have shown that $\Free(\CC)$ is an object of $\PrL{T}{\L}$.
		
		Next we will show that $\Free$ sends $S$-cocontinuous functors to $T$-cocontinuous functors. To this end fix a functor $L\colon \CC\rightarrow \CD$ in $\PrL{T}{S}$ and write $\mathbb{L}\colon \Unco{\CC}\rightarrow \Unco{\CD}$ for its unstraightening. Because the naturality squares in $L$ are left adjointable for maps in $\bbF_S$, \Cref{prop:adj_part_lax_lim} implies that $\mathbb{L}$ admits a right adjoint $\mathbb{R}\colon \Unco{\CD}\rightarrow \Unco{\CC}$ in $\Cat_{/\CI}$ which preserves cocartesian edges over $\bbF_S$. Now consider the description of $\Free(-)$ from \Cref{rem:other_description}. From this it is clear that postcomposition $\mathbb{R}$ gives a $T$-functor $\Free(\CD)\rightarrow \Free(\CC)$ which is right adjoint to $\Free(L)$. Therefore we conclude that $\Free(L)$ is a $T$-cocontinuous functor, see~\Cref{rem:PrL_the_same}. In total we have shown that $\Free\colon \PrL{T}{S}\rightarrow \Cat_T$ restricts to the subcategory $\PrL{T}{\L}$.
	\end{proof}
	
	Having shown that $\Free$ restricts appropriately, we now turn to showing that it is left adjoint to the forgetful functor $\mathrm{fgt}\colon \PrL{T}{S}\rightarrow \PrL{T}{\L}$. To do this we define the unit and counit of the putative adjunction.
	
	\begin{construction}
		Let $\CC\in \PrL{T}{S}$ be an $S$-presentable $T$-category. We define a $T$-functor
		\[I\colon \CC\rightarrow \Free(\CC)\] as follows. First observe that because each category $({\mathbb{F}_T}_{/X})^{\op}$ admits a final object we obtain a natural equivalence $\lim \pi_X^* \CC \simeq \CC_X$, given by evaluating at the final object. After identifying these two categories, we claim that including the limit into the partially lax limit $\lim \pi_X^* \CC\rightarrow \laxlimdag \pi_X^*\CC$ gives a natural $S$-cocontinuous $T$-functor $I\colon \CC\rightarrow \Free(\CC)$. 
		
		To see that $I$ is in fact $S$-cocontinuous we first note that by \Cref{prop:left_adj_to_res_in_lim}, each functor $I_X$ admits a right adjoint given by evaluating an object $s\colon {\mathbb{F}_T}_{/X}\rightarrow \Unco{\pi_X^* \CC}$ of $\laxlimdag \pi_X^* \CC$ at the object ${\id_X\colon X\rightarrow X}$ in ${\mathbb{F}_T}_{/X}$. Next we consider the left adjointability of naturality squares for maps in $\bbF_S$. By passing to total mates, it suffices to show that given any map $f\colon X\rightarrow Y$ in $\bbF_T$, the Beck--Chevalley transformation filling the square
		\[\begin{tikzcd}
			{\Free(\CC)_{X}} & {\CC_X} \\
			{\Free(\CC)_Y} & {\CC_Y}
			\arrow["{\alpha^*}", from=2-1, to=1-1]
			\arrow["{\alpha^*}"', from=2-2, to=1-2]
			\arrow["{\ev_{\id_Y}}"', from=2-1, to=2-2]
			\arrow["{\ev_{\id_X}}", from=1-1, to=1-2]
			\arrow[shorten <=11pt, shorten >=7pt, Rightarrow, from=2-2, to=1-1]
		\end{tikzcd}\] is an equivalence. One can compute that it is given at $s\colon {\bbF_{T}}_{/Y}\rightarrow \Unco{F}$ by the lax structure map $s_\alpha\colon \alpha^* s(\id) \rightarrow s(\alpha)$. Because the objects of $\Free(\CC)_X$ are strict on $\bbF_S\subset \bbF_T$, we conclude that this is an equivalence whenever $\alpha$ is a map in $\bbF_S$. We conclude that $I$ is a morphism in $\PrL{T}{S}$.
	\end{construction}
	
	\begin{construction}
		Next we construct the counit of the desired adjunction. Given an object $\CC$ in $\PrL{T}{\L}$ we have to construct a functor 
		\[L \colon \Free(\CC)\rightarrow \CC\] in $\PrL{T}{\L}$. We will do this by showing that the functor $I\colon\CC\rightarrow \Free(\CC)$ has a parametrized left adjoint when $\CC$ is $T$-presentable. To do this we first observe that by \Cref{prop:lim_in_part_laxlim}, $I_X$ preserves all limits. Therefore each functor $I_X$ has a left adjoint $L_X$ by another application of the adjoint functor theorem.
		
		To show that $I$ in fact has a left adjoint as a $T$-functor, it now suffices by \cite{MW21}*{Lemma 3.2.7} to show that the Beck--Chevalley transformation
		\[\begin{tikzcd}
			{\CC_X} & {\Free(\CC)_X} \\
			{\CC_Y} & {\Free(\CC)_Y}
			\arrow["{f_*}", from=1-2, to=2-2]
			\arrow["{I_X}", from=1-1, to=1-2]
			\arrow["{I_Y}"', from=2-1, to=2-2]
			\arrow["{f_*}"', from=1-1, to=2-1]
			\arrow[shorten <=11pt, shorten >=9pt, Rightarrow, from=2-1, to=1-2]
		\end{tikzcd}\]
		filling the square above is an equivalence for all $f\colon X\rightarrow Y$ in $\mathbb{F}_T$. Unwinding the definition of the Beck--Chevalley transformation, we find that it is given at $h\colon Z\rightarrow Y$ by the Beck--Chevalley transformation filling the square
		\[\begin{tikzcd}
			{\CC_X} &{\CC_{X\times_Y Z}} \\
			{\CC_Y} & {\CC_Z} 
			\arrow["{h^*}"', from=2-1, to=2-2]
			\arrow["{(\pi_1)_*}", from=1-2, to=2-2]
			\arrow["{(h\times_Y X)^*}", from=1-1, to=1-2]
			\arrow["{f_*}"', from=1-1, to=2-1]
			\arrow[shorten <=9pt, shorten >=9pt, Rightarrow, from=2-1, to=1-2]
		\end{tikzcd}\]
		This is an equivalence because $\CC$ is an object of $\PrL{T}{\L}$ and so satisfies base-change with respect to all pullback squares in $\bbF_T$. In total we conclude that $L$ is a functor in $\PrL{T}{\L}$.
	\end{construction}
	
	We will prove that $\Free$ is a left adjoint by showing that both triangles identities hold for the putative unit and counit. To do this it will be useful to have a different description of the composite $\Free\Free(\CC)$.
	
	\begin{lemma}\label{lem:iterated_free}
	There exists an equivalence 
	\[\Free(\Free(\CC))_X \simeq \laxlimdag_{\Ar({\bbF_T}_{/X})^{\op}} \pi^*_X\CC \circ s,\] 
	where $s\colon \Ar({\bbF_T}_{/X})^{\op}\rightarrow ({\bbF_T}_{/X})^{\op}$ is the source projection and $\Ar({\bbF_T}_{/X})^{\op}$ is marked by those natural transformations for which both maps are in $\bbF_S$. Moreover this equivalence is natural in both $\CC$ and $X$.
	\end{lemma}
	
	\begin{proof}
	This result follows immediately from \Cref{prop:lim_on_colim_diagram}, combined with the well-known computation that $\Ar({\bbF_T}_{/X})^{\op} \rightarrow ({\bbF_T}_{/X})^{\op}$ is equivalent to the cocartesian unstraightening of the slice functor $({\bbF_T}_{/X})_{/\bullet}^{\op}$.
	\end{proof}
	
	\begin{theorem}\label{thm:Free_colimit}
	The functor $\Free\colon \PrL{T}{S}\rightarrow \PrL{T}{\L}$ is left adjoint to $\mathrm{fgt}\colon  \PrL{T}{\L}\rightarrow \PrL{T}{S}$.
	\end{theorem}

	\begin{proof}
	We show that $I$ and $L$ satisfy the triangle identities, and so are a unit and counit exhibiting $\Free\dashv \mathrm{fgt}$ as an adjunction. First we consider the composite 
	\[\CC\xrightarrow{I} \Free(\CC) \xrightarrow{L} \CC.\] Recall that $I_X$ is a fully faithful right adjoint to $L_X$. Therefore the counit gives a natural equivalence from the composite to the identity.
	For the other triangle identity we consider the composite
	\[\Free(\CC)\xrightarrow{\Free(I)} \Free(\Free(\CC)) \xrightarrow{L} \Free(\CC).\] We may equivalently show that the composite
	\[\Free(\CC)_X\xleftarrow{\Free(\ev_{\id})} \Free(\Free(\CC))_X \xleftarrow{I} \Free(\CC)_X\] given by passing to right adjoints is naturally homotopic to the identity. At this point we apply \Cref{lem:iterated_free} to rewrite $\Free(\Free(\CC))_X$ as the partially lax limit of $\pi^*_X\CC \circ s$ over $\Ar({\bbF_T}_{/X}$. One can easily show that under this identification $I$ and $\Free(\ev_{\id})$ are given by restricting along the source projection $s\colon \Ar({\bbF_T}_{/X})^{\op}\rightarrow ({\bbF_T}_{/X})^{\op}$ and the identity section $c\colon ({\bbF_T}_{/X})^{\op}\rightarrow \Ar({\bbF_T}_{/X})^{\op}$ respectively. Therefore there is a natural equivalence between the composite $I\circ\Free(\ev_{\id})$ and restriction along $s\circ c = \id$.
	\end{proof}
	
	As an example of the process of freely adding colimits we obtain the following result:
	
	\begin{corollary}\label{cor:Glob_P_spc}
		$\Free(\Spc^S_\bullet)\simeq \Spc^T_\bullet$.
	\end{corollary}
	
	\begin{proof}
		This follows immediately by comparing universal properties: by \Cref{ex:S-spaces} both $\Free(\Spc^S_\bullet)$ and $\Spc^T_\bullet$ represent the functor $\CC \mapsto \core(\Gamma\CC)$.
	\end{proof}
	
	We note that the proof of the previous theorem did not require any knowledge about the left adjoint of $f^*\colon \Free(\CC)_Y \rightarrow \Free(\CC)_X$, beyond their existence. In fact the author does not know a general explicit description of the left adjoint of $f^*\colon \Free(\CC)_Y\rightarrow \Free(\CC)_X$ for an arbitrary map $f\colon X\rightarrow Y$. Nevertheless, we now show that when $f$ is in $\bbF_S$ such a description is in fact possible. Therefore, rather fittingly, it is only the left adjoints which we have freely adjoined which will remain mysterious. For this we require the concept of marked finality.
	
	\begin{definition}
		A functor $F\colon\CI\rightarrow \CJ$ of marked categories is marked final if for every functor $G\colon \CJ\rightarrow \Cat$, restriction along $F$ induces an equivalence
		\[
		\laxlimdag_{\CJ} G \rightarrow \laxlimdag_{\CI} GF.
		\]
	\end{definition}
	
	The following criteria allows us to recognize marked final functors. Before stating it we recall some notation. Given a functor $F\colon \CI\rightarrow \CJ$ and an object $j\in \CJ$, we write $\CI_{/j}$ for the comma category $F\downarrow \{j\}$. If $\CJ$ is marked, then we enhance this to a marked category by marking all the edges whose projection to $\CJ$ is marked. Furthermore given a marked category $\CJ$ we write $\mathcal{L}(\CJ)$ for the (Dwyer--Kan) localization of $\CJ$ at the marked edges.
	
	\begin{proposition}[\cite{AG20}*{Proposition 5.6$^{\op}$}]\label{prop:marked_finality}
		$F\colon \CI\rightarrow \CJ$ is marked final if and only if for all $j\in \CJ$ the canonical functor $F\colon \CI_{/j}\rightarrow \CJ_{/j}$ induces an equivalence
		\[
		\CL(\CI_{/j})\xrightarrow{\sim} \CL(\CJ_{/j})
		\]
		after localization.\qed
	\end{proposition}
	
	\begin{proposition}\label{prop:marked_adj}
		Let $\CI$ and $\CJ$ be marked categories. Suppose $L\colon \CI\rightleftarrows \CJ\cocolon R$ is an adjoint pair such that both $L$ and $R$ preserve the marking. Then the following are equivalent:
		\begin{enumerate}
			\item The unit $\eta\colon i \rightarrow RL(i)$ is marked for all $i\in \CI$;
			\item The counit $\epsilon \colon LR(j)\rightarrow j$ is marked for all $j\in \CJ$;
			\item The adjunction equivalence
			\[
			\Hom_{\CI}(L(i),j) \simeq \Hom_{\CJ}(i,R(j))
			\]
			preserves marked morphisms.
		\end{enumerate}
	\end{proposition}
	
	\begin{proof}
		Since identities are always marked, (3) clearly implies (1) and (2). Let us now show that (3) implies (1). Recall that the adjunction equivalence is given by the following composite 
		\[
		\Hom_{\CI}(L(i),j) \rightarrow \Hom_{\CJ}(RL(i),R(j))\xrightarrow{\eta^*} \Hom_{\CJ}(i,R(j))
		\] The first map preserves marked morphisms because $R$ was assumed to be a marked functor, and the second because the unit is marked and marked morphisms form a subcategory. That (3) implies (2) is similar.
	\end{proof}
	
	\begin{definition}
		We say $L\colon \CI\rightleftarrows \CJ\cocolon R$ is a marked adjunction if the equivalent conditions of the previous proposition holds.
	\end{definition}
	
	\begin{proposition}\label{prop:marked_adj_is_final}
		Let $L\colon \CI\rightleftarrows \CJ\cocolon R$ be a marked adjunction, then $L$ is marked final.
	\end{proposition}
	
	\begin{proof}
		Let $j\in \CJ$ and consider the functor
		\[
		L\colon \CI_{/j}\rightarrow \CJ_{/j}.
		\]
		This admits a right adjoint given by sending $f\colon j'\rightarrow j$ to the pair \[(R(j'),LR(j')\xrightarrow{f} LR(j)\xrightarrow{\epsilon} j).\] One computes that the unit and counit are given by the maps 
		\[
		\eta\colon i\rightarrow RL(i) \quad \text{and} \quad \epsilon\colon  LR(i)\rightarrow i.
		\]
		respectively. In particular both are marked by \Cref{prop:marked_adj}. We conclude that after localizing at the marked morphisms this adjunction is an adjoint equivalence, and so we conclude by \Cref{prop:marked_finality}.
	\end{proof}
	
	\begin{proposition}\label{prop:precom_marked_final}
		Let $(T,S)$ be an orbital pair and let $f\colon X\rightarrow Y$ be a map in $\bbF_S$. Then 
		\[
		f_!\colon {\bbF_T}_{/X} \rightleftarrows {\bbF_T}_{/Y}\cocolon f^*
		\] is a marked adjunction, where as always both categories are marked by those morphisms which lie in $\bbF_S$.
	\end{proposition}
	
	\begin{proof}
		We have previously observed already that both $f_!$ and $f^*$ preserve marked edges. The counit of the adjunction is given on an object $Z\rightarrow Y$ by the map $\pi_1$ in the pullback square 
		\[\begin{tikzcd}
			{X\times_Y Z} & Z \\
			X & Y
			\arrow["{\pi_1}", from=1-1, to=1-2]
			\arrow[from=1-2, to=2-2]
			\arrow["f", from=2-1, to=2-2]
			\arrow[from=1-1, to=2-1]
			\arrow["\lrcorner"{anchor=center, pos=0.125}, draw=none, from=1-1, to=2-2]
		\end{tikzcd}\]
		In particular as a pullback of $f$ it is again in $\bbF_S$.
	\end{proof}
	
	\begin{construction}\label{const:left_adj_to_S}
		We will now give a description of the left adjoint of the restriction functor $f^*\colon \Free(\CC)_Y \rightarrow \Free(\CC)_X$ when $f$ is in $\bbF_S$. First we note that to simplify notation we may pass to slices and assume that $Y$ is the final object. Now recall that the restriction functor $f^*\colon \Free(\CC)_Y \rightarrow \Free(\CC)_X$ is given by the functor 
		\[
		(f_!)^*\colon \laxlimdag_{{\bbF_T^{\op}}} \CC\rightarrow \laxlimdag_{({\bbF_T}_{/X})^{\op}} f^* \CC.
		\]
		To understand the left adjoint of this functor we may postcompose by the functor
		\begin{equation}\label{eq:pullback_eq}
			(f^*)^*\colon \laxlimdag_{({\bbF_T}_{/X})^{\op}} f^*\CC\rightarrow \laxlimdag_{{\bbF_T^{\op}}} f_*f^* \CC,
		\end{equation} which is an equivalence by combining \Cref{prop:marked_adj_is_final} and \Cref{prop:precom_marked_final}, and instead construct a left adjoint of the composite $(f_!f^*)^*$. This functor can again be reinterpreted. Note that the counit transformation $\epsilon\colon f_! f^* \Rightarrow \id$ induces a natural transformation from the identity on $\Fun^\dagger_{\mathbb{F}_T^{\op}}(\id,p)$ to $(f_!f^*)^*$. Evaluating this natural transformation on a section $s$ in $\laxlimdag \CC$ we find that $\gamma$ is pointwise cocartesian: at an object $Z\in \bbF_T$, the natural transformation $\gamma\colon s\rightarrow s\circ f_!f^*$ is given by applying $s$ to the map $\pi_1\colon X\times_Y Z\rightarrow Z$, which as a pullback of $f$ is in $\bbF_S$. This implies that when restricted to  $\laxlimdag \CC$, the functor $(f_! f^*)^*$ is naturally equivalent to cocartesian pushforward along the counit $\epsilon\colon f_! f^* \Rightarrow \id$. However this is in turn equivalent to postcomposition by the functor $\ul{f}^*\colon \Unco{\CC}\rightarrow \Unco{f_*f^*\CC}$. We conclude that after applying the equivalence \ref{eq:pullback_eq}, $f^*\colon \Free(\CC)_Y\rightarrow \Free(\CC)_X$ is homotopic to the functor
		\[
		\laxlimdag_{{\bbF_T^{\op}}} \CC \coloneqq \Fun_{{\bbF_T^{\op}}}({\bbF_T^{\op}},\Unco{\CC})\xrightarrow{(\ul{f}^*)_*} \Fun_{{\bbF_T^{\op}}}({\bbF_T^{\op}},\Unco{f_*f^*\CC}) \simeq \laxlimdag_{{\bbF_T^{\op}}} f_*f^*\CC.
		\] 
		This functor admits an explicit left adjoint. Namely, $\ul{f}^*\colon \Unco{\CC}\rightarrow \Unco{f_*f^*\CC}$ admits a relative left adjoint $\ul{f}_!$ given by the unstraightening of the parametrized left adjoint from \Cref{rem:internal_char_of_colims}. By \Cref{prop:adj_part_lax_lim}, postcomposition by $f_!$ defines a left adjoint to $f^*$. 
	\end{construction}
	
	\subsection{Globalization}
	
	We are most interested in the previous results when $T$ is $\Glo$ and $S$ is $\Orb$. As an example, we find that in this case \Cref{cor:Glob_P_spc} implies that $\FreeOrb(\Spc_{\bullet}) \simeq \Spcbulgl$, where we use the notation of \Cref{ex:eq_spaces}. 
	
	\begin{remark}
		Evaluating at $\bB G$ we find that in particular \[\laxlimdag_{(\Glo_{/G})^{\op}} \Spc_{\bullet} \simeq \Spc_{G\text{-}\gl}.\] This was proven for $G=e$ via different methods in the case of an arbitrary family of compact Lie groups as \cite{LNP}*{Theorem 6.18}.
	\end{remark}
	
	We have shown that the $\Orb$-relative cocompletion of the global category of equivariant spaces is given by the global category of globally equivariant spaces. In other words, in this case $\FreeOrb$ sends a global category of ``equivariant objects" to a global category of ``globally equivariant objects". Another example of this phenomena is given by equivariant spectra, whose relative cocompletion is given by globally equivariant spectra as we will show. For this reason we introduce the following notation.
	
	\begin{notation}
		We will refer to $\FreeOrb(\CC)$ as the \textit{globalization} of $\CC$ and denote it by $\Glob(\CC)$. 
	\end{notation}
	
	\section{Globalization and equivariant stability}\label{sec:P-stable}
	
	We will now lead up to a proof of \Cref{intro_thm_B} of the introduction. To do this we begin by recalling the notion of $P$-semiadditivity and $P$-stability for $T$-categories introduced in \cite{CLL23}. When $(T,P)=(\Glo,\Orb)$ we obtain the notions of equivariant semiadditivity and equivariant stability for global categories. We then recall the main results of \cite{CLL23,CLL_Partial}, which identify the universal globally presentable and equivariantly presentable equivariantly stable global categories with globally equivariant spectra and equivariant spectra respectively.
	
	Finally, as the new results of this section, we show that for a orbital pair $(T,S)$, $\CP^T_S(-)$ preserves $P$-semiadditivity and $P$-stability whenever $P$ is a subcategory of $S$. Combining this with the main results of \cite{CLL23,CLL_Partial} we conclude \Cref{intro_thm_B}, which identifies the global category of globally equivariant spectra as the globalization of the global category of equivariant spectra.
	
	\subsection{Recollection}
	
	We begin with a recollection of the relevant material from \cite{CLL23}. 
	
	\begin{definition}
		An orbital subcategory $P\subset T$ is called an \emph{atomic orbital subcategory} if every map in $P$ that admits a section in $T$ is an equivalence.
	\end{definition}
	
	Throughout this section, we fix an orbital pair $(T,S)$ and an atomic orbital subcategory $P$ of $T$ such that $P\subset S$.
	
	\begin{definition}
		We say an $S$-presentable $T$-category $\CC$ is \emph{pointed} if for all $X\in \bbF_{T}$, $\CC_X$ is pointed. We define $\PrLast{T}{S}$ to be the full subcategory of $\PrL{T}{S}$ spanned by the pointed global categories.
	\end{definition}
	
	\begin{construction}\label{const:dual_adj_norm}
		Let $\CC$ be a pointed $S$-presentable $T$-category. For any map $p\colon A\to B$ in $\mathbb F_P \subset \mathbb F_T$, \cite{CLL23}*{Construction 4.6.1} defines an \emph{adjoint norm map}
		\[
		\Nmadjdual_{p}\colon p^*p_! \Rightarrow \id.
		\]
	\end{construction}
	
	\begin{definition}
		A $S$-presentable $T$-category $\mathcal C$ is called \emph{$P$-semiadditive} if it is pointed and the adjoint norm map $\Nmadjdual_{p}\colon p^*p_! \Rightarrow \id$ is a counit transformation exhibiting $p^*$ as a left adjoint of $p_!$ for every $p\in \bbF_P$.
	\end{definition}
	
	\begin{remark}\label{rem:parametrized_norm}
		This definition is equivalent to \cite{CLL23}*{Definition 4.5.1} by Lemma 4.5.4 of \textit{op.~cit.} Furthermore one can show that an $S$-presentable $T$-category $\CC$ is $P$-semiadditive if and only if it is pointed, and for all $p\colon X\rightarrow Y$ in $\bbF_P$ a natural transformation $\Nmadjdual_p \colon \ul{p}^*\ul{p}_!\Rightarrow \id$, defined analogously to \Cref{const:dual_adj_norm}, is a counit transformation exhibiting $\ul{p}_!$ as a right adjoint of $\ul{p}^*$.
	\end{remark}
	
	\begin{example}
		When $P\subset T$ equals $\Orb_G\subset \Orb_G$, the notion of semiadditivity obtained agrees with $G$-semiadditivity as defined in \cite{nardin2017thesis}, see \cite{CLL23}*{Proposition 4.6.4}.
	\end{example}
	
	\begin{definition}
		We write $\PrLPsemi{T}{S}$ for the full subcategories of $\PrL{T}{S}$ spanned by the $P$-semiadditive $T$-categories.
	\end{definition}
	
	We may additionally impose a fiberwise stability condition.
	
	\begin{definition}
		We say a $S$-presentable $T$-category $\mathcal C$ is \emph{fiberwise stable} if $\CC_X$ is stable for all $X\in \bbF_T$. We say a $S$-presentable $T$-category $\mathcal C$ is \emph{$P$-stable} if it is $P$-semiadditive and fiberwise stable. We write $\PrLPst{T}{S}$ for the full subcategory of $\PrL{T}{S}$ spanned by the $P$-stable $T$-categories.
	\end{definition}
	
	We also specialize the notions above to the setting of global categories.
	
	\begin{definition}[\cite{CLL23}]
		We say an equivariantly presentable global category $\CC$ is equivariantly semiadditive or equivariantly stable if it is is $\Orb$-semiadditive or $\Orb$-stable respectively. 
	\end{definition}
	
	The main results of \cite{CLL23} and \cite{CLL_Partial} allow us to identify the free equivariantly presentable and globally presentable equivariantly stable global categories on a point.
	
	\begin{definition}
		We define $\Spbulgl$, the global category of \textit{globally equivariant spectra}, to be diagram which sends $\CB G$ to the category of \textit{G-global spectra}, in the sense of \cite{LenzGglobal}. This in turn is defined to be the localization of the category of symmetric $G$-spectra at the $G$-global weak equivalences. See \cite{CLL23}*{Section 7.1} for precise definitions.
	\end{definition}
	
	\begin{theorem}[\cite{CLL23}*{Theorem 7.3.2}]\label{thm:Sp_gl_univ}
		$\Spbulgl$ is the free globally presentable equivariantly stable global category on a point. That is, given any globally presentable equivariantly stable global category $\CC$, evaluating at the global sphere spectrum $\mathbb{S}_{\gl}\in \Sp_{\gl}$ gives an equivalence 
		\[
		\ul{\Fun}^{\L}(\Spbulgl,\CC) \simeq \CC,
		\] where the left hand side denotes the global category of cocontinuous functors. Evaluating at $\CB e$ and taking groupoid cores we obtain an equivalence
		\[
		\Hom_{\PrL{\Glo}{\L}}(\Spbulgl,\CC) \simeq \core(\CC(\CB e)).
		\] 
	\end{theorem}
	Similarly we have an equivariantly presentable version of the previous theorem.
	
	\begin{theorem}[\cite{CLL_Partial}*{Theorem 9.4}]\label{thm:Sp_univ}
		$\Sp_\bullet$ is the free equivariantly presentable equivariantly stable global category on a point. That is, given any equivariantly presentable equivariantly stable global category $\CC$, evaluation at the sphere spectrum $\mathbb{S}\in \Sp$ gives an equivalence
		\[
		\ul{\Fun}^{\mathrm{eq}\text{-}\mathrm{cc}}(\Sp_\bullet,\CC) \simeq \CC,
		\] where the left hand side denotes the global category of equivariantly cocontinuous functors. Evaluating at $\CB e$ and taking groupoid cores we obtain an equivalence
		\[
		\Hom_{\PrL{\Glo}{\Orb}}(\Sp_\bullet,\CC) \simeq \core(\CC(\CB e)).
		\]
	\end{theorem}
	
	\subsection{Globalizing equivariantly semiadditive and stable categories}
	
	We now show that relative cocompletion preserves $P$-semiadditivity and $P$-stability. We then conclude \Cref{intro_thm_B}.
	
	\begin{proposition}\label{prop:rel_cocomp_semi}
		Let $(T,S)$ be an orbital pair and suppose $P$ is an atomic orbital subcategory of $T$ such that $P\subset S$. The functor $\mathcal{P}^T_S\colon \PrL{T}{S}\rightarrow \PrL{T}{\L}$ restricts to functors
		\[
		\PrLPsemi{T}{S}\rightarrow \PrLPsemi{T}{\L} \quad \text{and} \quad \PrLPst{T}{S}\rightarrow \PrLPst{T}{\L}.
		\]
	\end{proposition}
	
	\begin{proof}
		Let $\CC$ be a $P$-semiadditive $S$-presentable $T$-category. First we note that since limits and colimits are computed pointwise in $\Free(\CC)_X$ for all $X\in \bbF_T$, it is again pointed. Now consider $p\colon X\rightarrow Y$ in $\mathbb{F}_P$. By passing to slices we may assume $Y$ is the final object of $\bbF_T$. We have to show that the adjoint norm map $\Nmadjdual_p\colon p^*p_!\Rightarrow \id$ is the counit of an adjunction. However recall that by \Cref{const:left_adj_to_S} the adjunction $p_! \dashv p^*$ can be identified up to equivalence with the adjunction 
		\[
		(\ul{p}_!)_*\colon \laxlimdag \CC \rightleftarrows \laxlimdag \CC\circ p_!p^* \cocolon (\ul{p}^*)_*.
		\]
		However because $\CC$ is $P$-semiadditive,  By \Cref{rem:parametrized_norm} there exists a natural transformation $\Nmadjdual \colon \ul{p}^*\ul{p}_!\Rightarrow \id$ which is the counit of an adjunction. By the 2-functoriality of $\Fun_{\bbF_T}^\dagger(\bbF_T, -)$, this induces a counit witnessing $(\ul{p}_!)_*$ as a right adjoint to $(\ul{p}^*)_*$. A tedious diagram chase shows that this transformation agrees, after applying suitable equivalences, with the adjoint norm map $\Nmadjdual_p$ of $\Free(\CC)$. We conclude that $\mathcal{P}^T_S(\CC)$ is $P$-semiadditive.
		
		Finally we note that because colimits and limits in $\mathcal{P}^T_S(\CC)$ are computed pointwise, $\mathcal{P}^T_S$ clearly preserves fiberwise stable global categories.
	\end{proof}
	
	Applying this in the global context we obtain the main theorem of this section.
	
	\begin{theorem}
		There is an equivalence 
		\[
		\Spbulgl\simeq \Glob(\Sp_{\bullet}).
		\]
	\end{theorem}
	
	\begin{proof}
		Note that by \Cref{prop:rel_cocomp_semi}, $\Glob(\Sp_\bullet)$ is again equivariantly stable. Therefore the result follows immediately from \Cref{thm:Sp_gl_univ} and \Cref{thm:Sp_univ} by comparing universal properties.
	\end{proof}
	
	As an immediate corollary we obtain a description of $G$-global spectra in the sense of \cite{LenzGglobal} as a partially lax limit for every finite group $G$.
	
	\begin{corollary}
		Let $G$ be a finite group. Then there is an equivalence
		\[
		\Sp_{G\text{-}\gl}\simeq \laxlimdag_{(\Glo_{/G})^{\op}} \Sp_\bullet.\qednow
		\]
	\end{corollary}
	
	\begin{remark}
		This result was proven for arbitrary families of compact Lie groups when $G$ is the trivial group as \cite{LNP}*{Theorem 11.10}. We expect that the result is true for $G$-global spectra defined with respect to an arbitrary compact Lie group $G$ and arbitrary family $\CF$.
	\end{remark}
	
	\begin{remark}
		There are equivariantly semiadditive analogues of the statements above. To avoid testing the readers patience we summarize them in this remark. One can define the global categories $\GammaS_{\bullet\text{-}\gl}^\textup{spc}$ and $\GammaS_\bullet^\textup{spc}$ of special global $\Gamma$-spaces and special equivariant $\Gamma$-spaces respectively. Evaluating these global categories at the groupoid $\CB G$ one obtains the category of special $G$-global $\Gamma$-spaces and the category of special $\Gamma$-$G$-space, as defined in \cite{LenzGglobal} and \cite{shimakawa} respectively.
		
		By \cite{CLL23}*{Theorem 5.3.1}, $\GammaS_{\bullet\text{-}\gl}^\textup{spc}$ is the free globally presentable equivariantly semiadditive global category on a point, while $\GammaS_\bullet^\textup{spc}$ is the free equivariantly presentable equivariantly semiadditive on a point by \cite{CLL_Partial}*{Theorem 7.17}. Therefore by comparing universal properties we obtain an equivalence
		\[
		\GammaS_{\bullet\text{-}\gl}^\textup{spc} \simeq \Glob(\GammaS_{\bullet}^\textup{spc})
		\]
		of global categories. Evaluating at the groupoid $\CB G$, we find that 
		\[
		\GammaS_{G\text{-}\gl}^\textup{spc} \simeq \laxlimdag_{(\Glo_{/G})^{\op}} \GammaS_{\bullet}^\textup{spc}.
		\]
	\end{remark}
	
	\section{Globalization and $\Rep$-stability}\label{sec:Rep-stable}
	
	We now switch gears and prove a consequence of the fact that $\Spbulgl$ is the globalization of $\Sp_\bullet$. Namely, in this section we show that $\Spbulgl$ is the initial globally presentable global category on which the representation spheres act invertibly. This universal property of global spectra was first suggested by David Gepner and Thomas Nikolaus, see \cite{OWF}. An analogous universal property has since been proven in the context of global model categories by \cite{LS23}. Our strategy for proving this universal property is first to observe that the condition that representation spheres act invertibly in fact already makes sense for pointed equivariantly presentable global categories, and therefore first consider the analogous question in this context.
	
	\begin{lemma}\label{bullet_spaces_initial}
		Let $\CC$ be an object of $\PrL{\Glo}{\Orb}$. Then $\CC$ admits a unique colimit preserving tensoring by $\Spc_\bullet$. In particular there exists a canonical $T$-functor 
		\[
		\Spc_\bullet \times \CC\rightarrow \CC
		\] which preserves $\Orb$-colimits in each variable, in the sense of~\cite{MW22}*{Definition 8.1.1}. Similarly if $\CC$ is in $\PrLast{\Glo}{\Orb}$, then it is uniquely tensored over the global category of pointed  equivariant spaces $\Spc_{\bullet,\ast}$, defined by the assignment $\CB G\mapsto \Spc_{G,\ast}\coloneqq (\Spc_G)_\ast$. 
	\end{lemma}
	
	\begin{proof}
		By~\cite{MW22}*{Corollary 8.2.5} the category of $\Spc_\bullet$-cocomplete global categories $\Cat_{\Glo}^{\Orb\textup{-cc}}$ admits a symmetric monoidal structure such that $\Orb$-cocontinuous global functors $\CC\otimes \CD\rightarrow \CE$ are equivalent to global functors $\CC\times \CD\rightarrow \CE$ which are $\Orb$-cocontinuous in each variable. Moreover by~\cite{MW22}*{Remark 8.2.6} $\Spc_\bullet$ is the unit of this category, and therefore every object in $\PrL{\Glo}{\Orb}$ inherits a tensoring by $\Spc_\bullet$. This shows the first statement. The second statement follows analogously to \cite{HA}*{Proposition 4.8.2.11}.
	\end{proof}
	
	\begin{remark}\label{rem:Orb_pres_coh_tensoring}
		The tensoring of $\CC$ by $\Spc_\bullet$ implies in particular that each category $\CC_X$ is tensored by $\Spc_X$. Furthermore given a map $f\colon X\rightarrow Y$ in $\bbF_{\gl}\coloneqq \bbF_{\Glo}$, the functor  $f^*\colon \CC_Y\rightarrow \CC_X$ is canonically $\Spc_Y$-linear, where $\CC_X$ is tensored over $\Spc_Y$ by restricting along the functor $f^*\colon \Spc_Y\rightarrow \Spc_X$.
		
		An analogous statement holds for the tensoring of a pointed equivariantly presentable global categories $\CC$ by $\Spc_{\bullet,\ast}$.
	\end{remark}
	
	\begin{remark}\label{rem:iota_shriek_module_map}
		Now suppose that $f\colon X\rightarrow Y$ in $\bbF_{\Orb}$ is a faithful functor. Given a equivariantly presentable global category $\CC$, the left adjoint $f_!\colon \CC_X\rightarrow \CC_Y$  of $f^*$ canonically inherits the structure of an oplax $\Spc_Y$-linear functor. Contained in the statement that the tensoring preserves $\Orb$-colimits in each variable is the fact that $f_!$ with this oplax $\Spc_Y$-linear structure is in fact a strong $\Spc_Y$-linear functor. Informally this means that the projection formula holds. Given this one can compute that the tensoring of $\Spc_G$ on $\CC_G$ is given by colimit extending the assignment
		\[\Orb_G\times \CC_G\rightarrow \CC_G, \quad (\iota \colon H\hookrightarrow G,\,C)\mapsto \iota_!\iota^*(C).\] 
		Once again an analogous statement holds for pointed equivariantly presentable global categories.
	\end{remark}
	
	\begin{definition}\label{def:aug_rep}
		We define a parametrized subcategory $\bbS^{\rep}$ of $\Spc_\bullet$ by letting $\bbS^{\rep}_G$ be the full subcategory spanned by the representation spheres $S^V = V\cup \{\infty\}$, where $V$ is any finite dimensional orthogonal representation of $G$. 
	\end{definition}
	
	\begin{definition}
		We define $\PrLVst{\Glo}{\Orb}$ to be the full subcategory of $\PrLast{\Glo}{\Orb}$ spanned by the objects $\CC$ such that $S^V\otimes(-)\colon \CC_G\rightarrow \CC_G$ is an equivalence for all $\CB G\in \Glo$ and all $S^V\in \bbS^{\rep}_G$. We call such global categories \emph{$\Rep$-stable}.
	\end{definition}
	
	\subsection{Formal Inversions}
	
	Given a equivariantly presentable global category $\CC$ we would like to construct the initial global category under $\CC$ which is $\Rep$-stable. In other words, we would like to understand the process of (Rep-)stabilizing equivariantly presentable global categories. Just as stabilizing ordinary categories is given by inverting the action of $S^1$, the topological sphere, the stabilization of an equivariantly presentable global category will be obtained by inverting the action of the representation spheres pointwise. We first recall the relevant definitions.
	
	\begin{definition}
		Let $\CD \in \Mod_{\CC}(\PrLun)$ be a presentable category tensored over another presentable symmetric monoidal category $\CC$. Furthermore fix a collection of objects $S\in \CC$. We say $\CD$ is \textit{$S$-local} if for every $X\in S$ the functor 
		\[X\otimes- \colon \CD \rightarrow \CD\]
		is an equivalence. We write $\Mod_{\CC}(\PrLun)^{S\textup{-loc}}$ for the full subcategory of $\Mod_{\CC}(\PrLun)$ spanned by the $S$-local objects.
	\end{definition}
	
	\begin{proposition}[\cite{Ro15}*{Proposition 4.10}]\label{prop:Robalo}
		The inclusion $\Mod_{\CC}(\PrLun)^{S\textup{-loc}}\subset \Mod_{\CC}(\PrLun)$ admits a symmetric monoidal left adjoint, which we denote by $\CD\mapsto \CD[S^{-1}]$. 
	\end{proposition}
	
	To invert the action of representation spheres pointwise in a global category, requires that inverting the action of a collection of objects is suitably functorial as the category we are tensored ober changes. To capture this functoriality we make the following definitions.
	
	\begin{definition}
		We define $\Cat_{\infty,\aug}$ to be the full subcategory of $\Fun([1],\Cat)$ spanned by the fully faithful functors $S\hookrightarrow \CC$ such that $S$ is a small category. The pair $(\CC,S)$ is called an augmented category.
	\end{definition}
	
	\begin{definition}
		We define the category $\CAlg(\PrLun)_{\aug}$ of augmented presentable symmetric monoidal categories as the following pullback:
		\[
		\begin{tikzcd}
			\CAlg(\PrLun)_{\aug} \arrow[r] \arrow[d] & \CAlg(\PrLun) \arrow[d] \\
			\Cat_{\infty,\aug} \arrow[r] &	\Cat_{\infty}
		\end{tikzcd}	
		\]
	\end{definition}
	
	\begin{definition}
		We write $\Mod(\PrLun)$ for the cartesian unstraightening of the functor 
		\[\Mod_{\bullet}(\PrLun)\colon \CAlg(\PrLun)^{\op}\rightarrow \Cat, \quad \CC \mapsto \Mod_\CC(\PrLun).\]
		Objects of $\Mod(\PrLun)$ consist of a pair $(\CC,\CD)$ of a presentable symmetric monoidal category $\CC$ and a presentable category $\CD$ tensored over $\CC$.
	\end{definition}
	
	\begin{definition}
		We define $\Mod(\PrLun)_{\aug}$ to be the pullback
		\[
		\begin{tikzcd}
			\Mod(\PrLun)_{\aug} \arrow[r] \arrow[d] & \Mod(\PrLun) \arrow[d] \\
			\CAlg(\PrLun)_{\aug} \arrow[r] &	\CAlg(\PrLun).
		\end{tikzcd}	
		\] 
		
		We define $\Mod(\PrLun)_{\aug^{-1}}$ to be the full subcategory of $\Mod(\PrLun)_{\aug}$ spanned by those triples $(\CC,S,\CD)$ such that for every $X\in S$, $X\otimes -\colon \CD\rightarrow \CD$ is an equivalence. 
	\end{definition}
	
	\begin{theorem}
		The inclusion $\Mod(\PrLun)_{\aug^{-1}}\hookrightarrow \Mod(\PrLun)_{\aug}$ admits a left adjoint
		\[\bbI\colon \Mod(\PrLun)_{\aug} \rightarrow \Mod(\PrLun)_{\aug^{-1}}.\] Furthermore this left adjoint sends a triple $(\CC,S,\CD)$ to the triple $(\CC,S,\CD[S^{-1}])$.
	\end{theorem}
	
	\begin{proof}
		Note that both $\Mod(\PrLun)_{\aug}$ and $\Mod(\PrLun)_{\aug^{-1}}$ are the total category of a cartesian fibration over $\CAlg(\PrLun)_{\aug}$: The first because it is a pullback of $\Mod(\PrLun)$. For the second observe that given an object $(\CC',T,\CD)\in \Mod(\PrLun)_{\aug^{-1}}$ and a morphism $F\colon (\CC,S)\rightarrow (\CC',T)$ in $\CAlg(\PrLun)_{\aug}$, an object $X\in S$ acts on $F^*\CD$ as $F(X)$ acts on $\CC$. Because the functor $F$ preserves the augmentation, $F(X)$ is in $T$ and so the action is invertible. This shows that $F^*(\CC,D,\CD)$ is again in $\Mod(\PrLun)_{\aug^{-1}}$. In other words, $\Mod(\PrLun)_{\aug^{-1}}$ is a full subcategory of $\Mod(\PrLun)_{\aug}$ closed under cartesian pushforward, and so a cartesian fibration again. 
		
		Note that this also shows that the inclusion $\Mod(\PrLun)_{\aug^{-1}}\rightarrow \Mod(\PrLun)_{\aug}$ is a map of cartesian fibrations. Given this, the statement is an application of \cite{HA}*{Proposition 7.3.2.6}, where the fiberwise left adjoints are given by \Cref{prop:Robalo}. The final statement is clear.
	\end{proof}
	
	The following proposition shows that $\bbI$ preserves products, in a suitable sense.
	
	\begin{proposition}\label{prop:inverting_objects_preserves_products}
		Consider a set of objects $(\CC_i,S_i,\CD_i)$ in $\Mod(\PrLun)_{\aug}$. Then there is an equivalence
		\[
		\bbI(\prod\CC_i,\prod S_i,\prod\CD_i)\simeq (\prod\CC_i,\prod S_i,\prod \CD_i[S_i^{-1}]).
		\]
	\end{proposition}
	
	\begin{proof}
		This follows immediately from the equivalences 
		\[\Mod_{\prod \CC_i}(\PrLun)\simeq \prod\Mod_{\CC_i}(\PrLun) \quad \text{and} \quad \Mod_{\prod \CC_i}(\PrLun)^{\prod S_i -\textup{loc}}\simeq \prod\Mod_{\CC_i}(\PrLun)^{S_i-\textup{loc}}.\qedhere\]
	\end{proof}
	
	\subsection{Equivariantly presentable Rep-stabilization}
	
	We are now ready to construct the $\Rep$-stabilization of an equivariantly presentable global category. First we note that the observations of \Cref{rem:Orb_pres_coh_tensoring} extend to a coherent statement:
	
	\begin{proposition}\label{prop:Orb_pres_coh_tensoring}
		Let $\CC$ be a pointed equivariantly presentable global category. The functor $\CC\colon \bbF_{\gl}^{\op}\rightarrow \PrLun$ extends to a functor \[\CC\colon \bbF_{\gl}^{\op}\rightarrow \Mod(\PrLun), \quad X\mapsto (\Spc_{X,\ast},\CC_X).\]
	\end{proposition}
	
	\begin{proof}
		By \Cref{bullet_spaces_initial} $\CC$ is canonically an object in $\Mod_{\Spc_{\bullet,\ast}}(\Cat_{\Glo})$. We note that $\Cat_{\Glo}$, as a functor category on $\Glo^{\op}$, is in particular an oplax limit of the constant $\Glo$-shaped diagram on $\Cat$. Therefore \cite{LNP}*{Theorem 5.10$^{\op}$} implies that $\CC$ gives an object in $\oplaxlim \Mod_{\Spc_{\bullet,\ast}} \Cat$. This is in turn equivalent to a functor $\Glo^{\op}\rightarrow \Mod(\Cat)$ by \cite{LNP}*{Theorem 4.13$^{\op}$}, where $\Mod(\Cat)$ is defined analogously to $\Mod(\PrLun)$. For this functor to factor through $\Mod(\PrLun)$ is a property, guaranteed by \Cref{bullet_spaces_initial}. We then limit extend this to a functor from $\bbF_{\gl}^{\op}$.
	\end{proof}
	
	\begin{remark}
		Note that \Cref{def:aug_rep} specifies a lift of the functor $\Spc_{\bullet,\ast}\colon \bbF_{\gl}^{\op}\rightarrow \CAlg(\PrLun)$ to a functor into $\CAlg(\PrLun)_{\aug}$. This in turn lifts the functor $\CC\colon \bbF_{\gl}^{\op}\rightarrow \Mod(\PrLun)$ of \Cref{prop:Orb_pres_coh_tensoring} to a functor into $\Mod(\PrLun)_{\aug}$.
	\end{remark}
	
	\begin{definition}\label{def:Stab_Orb}
		Let $\CC$ be a pointed equivariantly presentable global category. Postcomposing the lift of $\CC$ to a functor into $\Mod(\PrLun)_{\aug}$ with the functor $\bbI\colon \Mod(\PrLun)_{\aug}\rightarrow \Mod(\PrLun)_{\aug^{-1}}$ and then forgetting down to $\PrLun$ we obtain a new functor \[\Stab^{\Orb}(\CC)\colon \bbF_{\gl}^{\op}\rightarrow \PrLun,\quad X \mapsto \CC_X[(\mathbb{S}^{\rep}_X)^{-1}].\] By \Cref{prop:inverting_objects_preserves_products} this is again a global category. Therefore $\Stab^{\Orb}$ defines a functor $\PrLast{\Glo}{\Orb}\rightarrow \Cat_\Glo$.
	\end{definition}
	
	\begin{theorem}\label{thm:Stab_P}
		The functor $\Stab^{\Orb}$ lands in the subcategory $\PrLVst{\Glo}{\Orb}$, and is a left adjoint to the inclusion $\PrLVst{\Glo}{\Orb}\subset \PrLast{\Glo}{\Orb}$.
	\end{theorem}
	
	\begin{proof}
		The functor $\CC\rightarrow \Stab^{\Orb}(\CC)$ induced by the unit of $\bbI$ is by definition the initial natural transformation of left adjoints such that $\Stab^{\Orb}(\CC)$ is fiberwise presentable and the action of the representation spheres on $\Stab^{\Orb}(\CC)$ is invertible. Therefore it suffices to show that $\Stab^{\Orb}(\CC)$ is in $\PrL{\Glo}{\Orb}$ and that the extension of $F\colon \CC\rightarrow \CD$ to $F'\colon \Stab^{\Orb}(\CC)\rightarrow \CD$ preserves all equivariant colimits whenever $F$ does.
		
		First we show that for every map $\iota\colon X\rightarrow Y$ in $\bbF_{\Orb}$ the functor $\iota^*\colon \Stab^{\Orb}(\CC)_Y\rightarrow \Stab^{\Orb}(\CC)_X$ admits a left adjoint. For this the crucial input is the following property of $\bbS^{\mathrm{rep}}$: the restriction of the regular representation of $G$ to a subgroup $H$ is a multiple of the regular representation of $H$, and so every $H$-representation is a summand of the restriction of enough copies of the regular $G$-representation. This is obviously also true for more general maps $\iota\colon X\rightarrow Y$ in $\bbF_{\Orb}$. By \cite{Cnossen23}*{Lemma 2.22} we conclude that there exists an equivalence $\CC_X[(\bbS_Y^{\rep})^{-1}]\simeq \CC_X[(\bbS_X^{\rep})^{-1}]$ of $\Spc_{Y,\ast}$-modules. By \Cref{rem:iota_shriek_module_map} $\iota_!$ is a $\Spc_{Y,\ast}$-module map and so induces a functor
		\[\iota_!\colon \Stab^{\Orb}(\CC)_X\rightarrow \Stab^{\Orb}(\CC)_Y.\] Furthermore because $\iota_!\dashv \iota^*$ is an adjunction in $\CC_X$-modules, both the unit and counit are canonically $\Spc_{Y,\ast}$-linear natural transformations. These therefore also lift to natural transformations witnessing \[\iota_!\colon  \Stab^{\Orb}(\CC)_X\rightarrow \Stab^{\Orb}(\CC)_Y\] as a left adjoint to $\iota^*\colon \Stab^{\Orb}(\CC)_Y\rightarrow \Stab^{\Orb}(\CC)_X$ in $\Spc_{Y,\ast}[(\bbS^{\rep}_Y)^{-1}]$-modules.
		
		Next we show the required left adjointability conditions. Consider a square 
		
		\[\begin{tikzcd}
			{\CC_{Y'}[(\bbS^{\rep}_{Y'})^{-1}]} & {\CC_{X'}[(\bbS^{\rep}_{X'})^{-1}]} \\
			{\CC_Y[(\bbS^{\rep}_Y)^{-1}]} & {\CC_X[(\bbS^{\rep}_X)^{-1}]}
			\arrow["{g^*}", from=1-1, to=2-1]
			\arrow["{f^*}", from=1-2, to=2-2]
			\arrow["{j^*}", from=1-1, to=1-2]
			\arrow["{i^*}", from=2-1, to=2-2]
		\end{tikzcd}\]
		induced by a pullback square in $\bbF_{\gl}$ such that $i,j$ are faithful maps of groupoids, which we have to show is left adjointable. We first note that because all of the functors in this diagram preserve colimits, it suffices by~\cite{AI2023motivic}*{Lemma 1.5.1} to prove that the Beck-Chevalley transformation is an equivalence on objects of the form $S^{-V}\otimes Z$ for $Z\in \CC_{X'}$ and $S^{V}\in \bbS_{X'}^{\rep}$. Because $j^*$ induces a cofinal map between the augmentations we immediately see that it in fact suffices to prove this for objects of the form $S^{-j^* V}\otimes Z$, where $V$ is now a $Y'$-representation. One can show that the Beck-Chevalley transformation on $S^{-j^* V}\otimes Z$ is given by tensoring the Beck-Chevalley transformation of $X$ by $S^{-j^* V}$. More precisely we claim that the following diagram commutes:
		\[\hspace{-35.615pt}
		\begin{tikzcd}[column sep = large]
			{g^*j_!(S^{-j^*V}\otimes Z)} & {i_!i^*g^*j_!(j^*S^{-V}\otimes Z)} & {i_!f^*j^*j_!(j^*S^{-V}\otimes Z)} & {i_!f^*(S^{-j^*V}\otimes Z)} \\
			&&& {i_!(S^{-f^*j^*V}\otimes f^*Z)} \\
			{g^*(S^{-V}\otimes j_! Z)} &&& {i_!(S^{-i^*g^*V}\otimes f^*Z)} \\
			{S^{-g^*V}\otimes g^* j_!Z} & {S^{-g^* V} \otimes i_!i^*g^*j_! Z} & {S^{-g^* V}\otimes i_! f^* j^*j_!Z} & {S^{-g^* V}\otimes i_!f^* Z.}
			\arrow["\epsilon", from=1-3, to=1-4]
			\arrow["\sim", from=1-2, to=1-3]
			\arrow["\eta", from=1-1, to=1-2]
			\arrow["\sim"', from=1-1, to=3-1]
			\arrow["\sim"', from=3-1, to=4-1]
			\arrow["{S^{-g^* V} \otimes \,\eta}", from=4-1, to=4-2]
			\arrow["\sim", from=1-2, to=4-2]
			\arrow["\sim", from=1-3, to=4-3]
			\arrow["{S^{-g^* V}\otimes \,\sim}", from=4-2, to=4-3]
			\arrow["{S^{-g^* V}\otimes \,\epsilon}", from=4-3, to=4-4]
			\arrow["\sim", from=1-4, to=2-4]
			\arrow["\sim", from=3-4, to=4-4]
			\arrow["\sim", from=2-4, to=3-4]
		\end{tikzcd}
		\]
		The proof of this claim is a tedious diagram chase which we omit. From this we conclude that it suffices to check that the Beck-Chevalley transformation is an equivalence on objects in the image of the functor $\CC_{X'}\rightarrow \CC_{X'}[(\bbS^{\rep}_{X'})^{-1}]$. However on such objects the Beck-Chevalley transformation is simply given by the image of the Beck-Chevalley transformation for $\CC_\bullet$ and so an equivalence. In exactly the same way one shows that if $F\colon \CC\rightarrow \CD$ is a functor between pointed equivariantly presentable global categories which preserves $\Orb$-colimits then $\Stab^{\Orb}(F)$ again preserves $\Orb$-colimits.
	\end{proof}
	
	\begin{proposition}\label{prop:Stab_eq_is_spectra}
		$\Stab^\Orb(\Spc_\bullet) \simeq \Sp_\bullet$.
	\end{proposition}
	
	\begin{proof}
		By definition the diagram $\Sp_\bullet$ is given by pointwise stabilizing the diagram $\Spc_\bullet$ at the representation spheres, see~\Cref{ex:eq_spectra}.
	\end{proof}
	
	\subsection{Globally presentable Rep-stabilization}
	
	We now repeat the previous definitions for $\PrL{\Glo}{\L}$.
	
	\begin{definition}
		We say a globally presentable global category $\CC$ is pointed if each $\CC_G$ is pointed. Given a pointed globally presentable global category $\CC$ we say it is \textup{$\Rep$-stable} if it is $\Rep$-stable as an equivariantly presentable global category. We define $\PrLast{\Glo}{\L}$ and $\PrLVst{\Glo}{\L}$ to be the full subcategory of $\PrLast{\Glo}{\L}$ spanned by the pointed and $\Rep$-stable global categories respectively.
	\end{definition}
	
	\begin{definition}
		Consider a morphism $F\colon \CC\rightarrow \CC'$ in $\PrLast{\Glo}{\L}$. We say \textit{F exhibits $\CC'$ as the globally presentable $\Rep$-stabilization of $\CC$} if the map 
		\[
		\Hom_{\PrLast{\Glo}{\L}}(\CC',\CD)\rightarrow \Hom_{\PrLast{\Glo}{\L}}(\CC,\CD)
		\] is an equivalence for every $\CD\in \PrLVst{\Glo}{\L}$.
	\end{definition}
	
	The construction of globally presentable $\Rep$-stabilizations is significantly more subtle then the analogous story of equivariantly presentable $\Rep$-stabilizations. In particular the construction of globally presentable $\Rep$-stabilizations cannot be as simple as pointwise inverting the action of representation spheres. For example we remind the reader that while $\Spc_{\bullet}$ is globally presentable, $\Sp_\bullet$ is not. 
	
	Nevertheless, imitating arguments of \cite{Ro15}*{Section 2} one can show that globally presentable $\Rep$-stabilizations always exist. We will not consider the finer aspects of this construction, but instead content ourselves with the observation that when the globally presentable global category $\CC$ is presented as the globalization of some other global category, we can obtain an explicit description of $\Stab^\Glo(\CC)$.
	
	\begin{theorem}\label{thm:Stab_Glo_of_Free}
		Let $\CC$ be an object of $\PrLast{\Glo}{\Orb}$. The globally presentable global category $\Glob(\Stab^\Orb(\CC))$ is the globally presentable $\Rep$-stabilization of $\Glob(\CC)$.
	\end{theorem}
	
	\begin{proof}
		We claim that $\Glob(-)$ preserves $\Rep$-stable global categories. To this end fix a $\CD\in \PrLVst{\Glob}{\Orb}$. First note that because colimits and limits are computed pointwise in $\Glob(\CD)$, it is again a pointed global category. Now consider $S^V \in \bbS^\rep_G$. We compute that the tensoring of $S^V$ on $\Glob(\CD)_G$ is given by \[\{X_f\}_{f\colon \CB H\rightarrow \CB G} \mapsto \{f^*(S^V)\otimes X_f\}_{f\colon \CB H\rightarrow \CB G} \simeq  \{S^{f^*(V)}\otimes X_f\}_{f\colon \CB H\rightarrow \CB G}.\] Because each $S^{f^*(V)}$ acts invertibly on $\CD_H$, we conclude that tensoring by $S^V$ on $\Glob(\CD)_G$ is an equivalence. Therefore $\Glob$ restricts to a functor $\PrLVst{\Glob}{\Orb}\rightarrow \PrLVst{\Glo}{\L}$. The result now follows from the following series of natural equivalences:
		\begin{align*}
			\Hom_{\PrLVst{\Glo}{\L}}(\Glob(\Stab^\Orb(\CC)),\CD) &\simeq  \Hom_{\PrLVst{\Glo}{\Orb}}(\Stab^\Orb(\CC),\CD)\\
			&\simeq \Hom_{\PrLast{\Glo}{\Orb}}(\CC,\CD)\\
			&\simeq  \Hom_{\PrLast{\Glo}{\L}}(\Glob(\CC),\CD).
		\end{align*}
	\end{proof}
	
	\begin{corollary}
		The global category of globally equivariant spectra $\Spbulgl$ is the free globally presentable $\Rep$-stable global category on a point.
	\end{corollary}
	
	\begin{proof}
		Apply the previous result to $\CC = \Spc_\bullet$, using \Cref{prop:Stab_eq_is_spectra}.
	\end{proof}
	
	\begin{remark}
		We expect that both localizations $\PrL{\Glo}{\L}\rightarrow \PrLVst{\Glo}{\L}$ and $\PrL{\Glo}{\L}\rightarrow {\mathrm{Pr}^{\L}_{\Glo,\Orb\textup{-}\textup{st}}}$ are smashing. This would imply, among other things, that because they agree on the unit $\Spcbulgl$ they in fact agree as functors. In particular we would conclude that a globally presentable global category is equivariantly stable if and only if it is representation stable.
	\end{remark}
	
	\subsection{Symmetric monoidal structures}
	
	Observe that the partially lax limit of symmetric monoidal categories is canonically symmetric monoidal by taking the tensor product pointwise, see \cite{LNP}*{Section 3} for a more detailed discussion. This gives a lift of $\Spbulgl\colon \bbF_{\gl}^{\op}\rightarrow \PrLun$ to a functor into $\CAlg(\PrLun)$. We will write $\Spbulgl^{\otimes}$ for this functor. 
	
	The goal of this subsection is to show that $\Spbulgl^{\otimes}$ is the initial globally presentably symmetric monoidal $\Rep$-stable global category. To elaborate on this we recall that by~\cite{MW22}*{Proposition 8.2.9}, $\PrL{\Glo}{\L}$ admits a symmetric monoidal structure given by a parametrized version of the Lurie tensor product. As usual functors out of the tensor product corepresent functors out of the product which preserve global colimits in both variables. We call an object $\CC$ of the category $\CAlg(\PrL{\Glo}{\L})$ a globally presentably symmetric monoidal global category. To understand this it is useful to be more explicit about the structure and properties implicit in presentable symmetric monoidality. We will do this in the generality of an arbitrary orbital category $T$.
	
	\begin{definition}
		A symmetric monoidal $T$-category is a finite product preserving functor $\bbF_T^{\op}\rightarrow \Cat^\otimes$. We define $\Cat_T^\otimes$ to be the functor category $\Fun^{\times}(\bbF_T^{\op},\Cat^\otimes)$.
	\end{definition}
	
	\begin{definition}
		We define $\PrLotimes{T}{\L}$ to be the subcategory of $\Cat_T^\otimes$ spanned on objects by those $\CC$ such that 
		\begin{enumerate}
			\item the functor $\CC$ lifts to a functor $\CC\colon \bbF_T^{\op}\rightarrow \CAlg(\PrL{}{\L})$;
			\item for all $f\colon X \rightarrow Y$ in $\bbF_T$, the functor $f^*$ has a further left adjoint $f_!$;
			\item the square obtained by applying $\CC$ to a pullback square in $\bbF_T$ is left adjointable.
		\end{enumerate}
		On morphisms $\PrLotimes{T}{\L}$ is spanned by those functors $F\colon \CC\rightarrow \CD$ in $\Cat_T^\otimes$ such that each functor $F_G$ admits a right adjoint and the square
		\[\begin{tikzcd}
			{\CC_X} & {\CD_X} \\
			{\CC_Y} & {\CD_Y}
			\arrow["{F_X}", from=2-1, to=2-2]
			\arrow["{F_Y}", from=1-1, to=1-2]
			\arrow["{f^*}"', from=1-1, to=2-1]
			\arrow["{f^*}"', from=1-2, to=2-2]
		\end{tikzcd}\]
		is left adjointable for all $f\in T$. 
	\end{definition}
	
	\begin{proposition}\label{prop:pres_sym_mon}
		There exists a fully faithful forgetful functor $\CAlg(\PrL{T}{\L})\subset \PrLotimes{T}{\L}$, with essential image those $\CC\colon \bbF_T^{\op}\rightarrow \CAlg(\PrL{}{\L})$ which satisfy the left projection formula, i.e.~ those $\CC\in \PrLotimes{T}{\L}$ such that for all morphisms $f\colon X \rightarrow Y$ in $\bbF_T$, the canonical natural transformation 
		\[f_!(f^* X\otimes Y) \rightarrow f_!f^*X \otimes f_! Y \xrightarrow{\epsilon \otimes Y} X\otimes f_! Y\] is an equivalence.
	\end{proposition}
	
	\begin{proof}
		Both $\CAlg(\PrL{T}{\L})$ and $\PrLotimes{T}{\L}$ are subcategories of $\Cat_T^\otimes$, and therefore it suffices to compare the images. For this we note that an object $\CC\in \Cat^\otimes_T$ is in $\CAlg(\PrL{T}{\L})$ if and only if $\CC$ is presentable and the tensor product commutes with fiberwise and $T$-groupoid indexed colimits in each variable. The first two statements are equivalent to the claim that $\CC$ factors through $\CAlg(\PrLun)$, while the final statement is equivalent to the claim that the left projection formula holds.
	\end{proof}
	
	Write $\CAlg(\PrL{\Glo}{\L})_{\rep\text{-}\st}$ for the full subcategory of $\CAlg(\PrL{\Glo}{\L})$ spanned by those $\CC$ which are representation stable. Recall that our goal is to show that $\Spbulgl^{\otimes}$ is the initial object of this $\infty$-category. To do this we will take a slightly circuitous route. We define $\PrLVstotimes{\Glo}{\L}$ via the pullback
	\[\begin{tikzcd}
		{\PrLVstotimes{\Glo}{\L}} & {\PrLotimes{\Glo}{\L}} \\
		{\PrLVst{\Glo}{\L}} & {\PrL{\Glo}{\L}.}
		\arrow[""{name=0, anchor=center, inner sep=0}, from=2-1, to=2-2]
		\arrow[from=1-2, to=2-2]
		\arrow[from=1-1, to=2-1]
		\arrow[from=1-1, to=1-2]
		\arrow["\lrcorner"{anchor=center, pos=0.125}, draw=none, from=1-1, to=0]
	\end{tikzcd}\]
	
	\begin{lemma}
		The symmetric monoidal global category $\Spbulgl^\otimes$ admits a unique map to any object $\CD\in \CAlg(\PrL{\Glo}{\L})_{\rep\text{-}\st}$.
	\end{lemma}
	
	\begin{proof}
		This is proven by simply repeating all of the constructions and arguments in Sections \ref{sec:Glob} and \ref{sec:Rep-stable}, but now for objects $\CC\colon \bbF_{\gl}\rightarrow \PrLun$ which additionally lift to $\CAlg(\PrL{}{\L})$. For example one shows that the functor $\Free\colon \PrL{T}{S}\rightarrow \PrL{T}{\L}$ refines to a left adjoint $\PrLotimes{T}{S}\rightarrow \PrLotimes{T}{\L}$, where the definition of $\PrLotimes{T}{S}$ is analogous to that of $\PrLotimes{T}{\L}$. Because partially lax limits of symmetric monoidal categories are computed underlying, the proofs of \Cref{thm:Free_is_presentable} and \Cref{thm:Free_colimit} go through unchanged. Similarly for all other steps leading to the theorem. Given this, the results follows immediately from the fact that $\Spc_{\bullet}$ is an initial object of $\CAlg(\PrL{\Glo}{\L}))$, see \cite{MW22}*{Remark 8.2.6}.
	\end{proof}
	
	\begin{theorem}
		$\Spbulgl^\otimes$ is the initial globally presentable symmetric monoidal $\Rep$-stable global category.
	\end{theorem}
	
	\begin{proof}
		By the previous lemma, it suffices to prove that $\Spbulgl^\otimes$ is actually an object of $\CAlg(\PrL{\Glo}{\L})$. By \Cref{prop:pres_sym_mon} this amounts to proving that $\Spbulgl^\otimes$ satisfies the left projection formula. Pick a morphism $f\colon X\rightarrow Y$ in $\bbF_\gl$. Consider the functor $\Sigma^\infty_\bullet\colon \mathcal{S}_{\bullet\text{-}\gl,\ast}\rightarrow \Spbulgl$ exhibiting $\Spbulgl$ as the globally presentable $\Rep$-stabilization of $\mathcal{S}_{\bullet\text{-}\gl,\ast}$. The source is the unit of $\PrLast{\Glo}{\L}$ and so satisfies the left projection formula. Because $\Sigma^\infty_+$ is strong monoidal and commutes with global colimits, we conclude that suspension spectra in $\Sp_{X\text{-}\gl}$ and $\Sp_{Y\text{-}\gl}$ satisfy the left projection formula for $f$. That is, the map
		\[f_!(f^* E\otimes F) \rightarrow E\otimes f_! F\] is an equivalence when $E$ and $F$ are in the image of $\Sigma_+^\infty$. Now we note that the collection of objects which satisfy the left projection formula for $f$ is closed under desuspensions and colimits in both $\CC_X$ and $\CC_Y$. Because $\Spbulgl$ is generated as a fiberwise stable global category under fiberwise colimits by suspension spectra, we conclude that $\Spbulgl$ satisfies the left projection formula.
	\end{proof}
	
	\bibliographystyle{plain}
	\bibliography{reference}

\begin{bibdiv}
\begin{biblist}

\bib{AG20}{article}{
      author={Abell\'{a}n~Garc\'{\i}a, Fernando},
       title={Marked colimits and higher cofinality},
        date={2022},
        ISSN={2193-8407},
     journal={J. Homotopy Relat. Struct.},
      volume={17},
      number={1},
       pages={1\ndash 22},
         url={https://doi.org/10.1007/s40062-021-00296-2},
      review={\MR{4386461}},
}

\bib{AI2023motivic}{article}{
      author={Annala, Toni},
      author={Iwasa, Ryomei},
       title={Motivic spectra and universality of {$K$}-theory},
        date={2023},
     journal={arXiv:2204.03434},
}

\bib{exposeI}{article}{
      author={Barwick, Clark},
      author={Dotto, Emanuele},
      author={Glasman, Saul},
      author={Nardin, Denis},
      author={Shah, Jay},
       title={Parametrized higher category theory and higher algebra:
  Expos\'{e} {I} -- elements of parametrized higher category theory},
        date={2016},
     journal={arXiv:1608.03657},
}

\bib{BDS15}{article}{
      author={Balmer, Paul},
      author={Dell'Ambrogio, Ivo},
      author={Sanders, Beren},
       title={Grothendieck-{Neeman} duality and the {Wirthm{\"u}ller}
  isomorphism},
    language={English},
        date={2016},
        ISSN={0010-437X},
     journal={Compos. Math.},
      volume={152},
      number={8},
       pages={1740\ndash 1776},
}

\bib{Berman}{article}{
      author={Berman, John},
       title={On lax limits in infinity categories},
        date={2020},
     journal={arXiv:2006.10851},
}

\bib{CLL23}{article}{
      author={Cnossen, Bastiaan},
      author={Lenz, Tobias},
      author={Linskens, Sil},
       title={Parameterized stability and the universal property of global
  spectra},
        date={2023},
     journal={arXiv:2301.08240},
}

\bib{CLL_Partial}{article}{
      author={Cnossen, Bastiaan},
      author={Lenz, Tobias},
      author={Linskens, Sil},
       title={{Partial parametrized presentability and the universal property
  of equivariant spectra}},
        date={2023},
     journal={arXiv:2307.11001},
}

\bib{Cnossen23}{article}{
      author={Cnossen, Bastiaan},
       title={Twisted ambidexterity in equivariant homotopy theory},
        date={2023},
     journal={arXiv:2303.00736},
}

\bib{GH}{article}{
      author={Gepner, David},
      author={Henriques, Andre},
       title={Homotopy theory of orbispaces},
        date={2007},
     journal={arXiv:math.0701916},
}

\bib{gepner2020equivariant}{article}{
      author={Gepner, David},
      author={Meier, Lennart},
       title={On equivariant topological modular forms},
        date={2020},
     journal={arXiv:2004.10254},
}

\bib{HHLN1}{article}{
      author={Haugseng, Rune},
      author={Hebestreit, Fabian},
      author={Linskens, Sil},
      author={Nuiten, Joost},
       title={Lax monoidal adjunctions, two-variable fibrations and the
  calculus of mates},
        date={2023},
     journal={{Proceedings of the London Mathematical Society}},
      volume={127},
      number={4},
       pages={889\ndash 957},
}

\bib{LenzGglobal}{article}{
      author={Lenz, Tobias},
       title={{$G$}-{G}lobal homotopy theory and algebraic {K}-theory},
        date={2020},
     journal={arXiv:2012.12676},
}

\bib{LNP}{article}{
      author={Linskens, Sil},
      author={Nardin, Denis},
      author={Pol, Luca},
       title={Global homotopy theory via partially lax limits},
        date={2022},
     journal={arXiv:2206.01556},
}

\bib{LS23}{article}{
      author={Lenz, Tobias},
      author={Stahlhauer, Michael},
       title={Global model categories and topological {A}ndr\'e-{Q}uillen
  cohomology},
        date={2023},
     journal={arXiv:2302.06207},
}

\bib{HTT}{book}{
      author={Lurie, Jacob},
       title={Higher topos theory},
      series={Annals of Mathematics Studies},
   publisher={Princeton University Press, Princeton, NJ},
        date={2009},
      volume={170},
        ISBN={978-0-691-14049-0; 0-691-14049-9},
         url={https://doi.org/10.1515/9781400830558},
      review={\MR{2522659}},
}

\bib{HA}{article}{
      author={Lurie, Jacob},
       title={Higher algebra},
        date={2016},
     journal={Preprint, available at \url{https://www.math.ias.edu/~lurie/}},
}

\bib{MW21}{article}{
      author={Martini, Louis},
      author={Wolf, Sebastian},
       title={Limits and colimits in internal higher category theory},
        date={2021},
     journal={arXiv:2111.14495},
}

\bib{MW22}{article}{
      author={Martini, Louis},
      author={Wolf, Sebastian},
       title={Presentable categories internal to an $\infty$-topos},
        date={2022},
     journal={arXiv:2209.05103},
}

\bib{nardin2017thesis}{thesis}{
      author={Nardin, Denis},
       title={Stability and distributivity over orbital $\infty$-categories},
        type={Ph.D. Thesis},
        date={2017},
}

\bib{OWF}{article}{
      author={Nikolaus, Thomas},
       title={The universal property of global spectra and elliptic
  cohomology},
        date={2015},
     journal={Oberwolfach Reports},
      volume={12},
       pages={759\ndash 761},
}

\bib{Ro15}{article}{
      author={Robalo, Marco},
       title={{{\(K\)}}-theory and the bridge from motives to noncommutative
  motives},
    language={English},
        date={2015},
        ISSN={0001-8708},
     journal={Adv. Math.},
      volume={269},
       pages={399\ndash 550},
}

\bib{Schwede18}{book}{
      author={Schwede, Stefan},
       title={Global homotopy theory},
      series={New Mathematical Monographs},
   publisher={Cambridge University Press, Cambridge},
        date={2018},
      volume={34},
        ISBN={978-1-108-42581-0},
         url={https://doi.org/10.1017/9781108349161},
      review={\MR{3838307}},
}

\bib{shimakawa}{article}{
      author={{Shimakawa}, Kazuhisa},
       title={{Infinite Loop $G$-Spaces Associated to Monoidal $G$-Graded
  Categories}},
        date={1989},
     journal={{Publ. Res. Inst. Math. Sci.}},
      volume={25},
      number={2},
       pages={239\ndash 262},
}

\end{biblist}
\end{bibdiv}
\end{document}